\documentclass[11pt,reqno]{amsart}
\usepackage[utf8]{inputenc}
\usepackage{amsmath,amsfonts,amscd,amssymb,latexsym,mathtools}
\usepackage[left=2.9cm,right=2.9cm,top=2.66cm,bottom=2.66cm]{geometry}

\usepackage{mathrsfs}
\usepackage[percent]{overpic} 

\usepackage[nobysame]{amsrefs}
\usepackage{enumitem}
\usepackage{comment}
\usepackage{soul}
\usepackage{tikz}
\usepackage{marginnote}
\usepackage{bbm}
\usepackage{cancel,slashed}
\usepackage[normalem]{ulem}

\usepackage{footmisc}
\usepackage{hyperref}
\hypersetup{
    colorlinks=true,
    linkcolor=blue,
    citecolor=green,}

\numberwithin{equation}{section}

\newtheorem{theorem}{Theorem}[section]
\newtheorem{lemma}[theorem]{Lemma}
\newtheorem{proposition}[theorem]{Proposition}
\newtheorem{corollary}[theorem]{Corollary}

\theoremstyle{definition}
\newtheorem{definition}[theorem]{Definition}

\newtheorem{assumption}[theorem]{Assumption}

\theoremstyle{remark}
\newtheorem{remark}[theorem]{Remark}

\newcommand{\CC}{\mathbb{C}}

\newcommand{\RR}{\mathbb{R}}
\newcommand{\ZZ}{\mathbb{Z}}

\newcommand{\bfs}{\mathbf{s}}

\newcommand{\calD}{\mathcal{D}}

\newcommand{\calP}{\mathcal{P}}
\newcommand{\calR}{\mathcal{R}}

\newcommand{\frakt}{\mathfrak{t}}

\newcommand{\id}{\operatorname{id}}
\newcommand{\dom}{\operatorname{dom}}
\newcommand{\wid}{\operatorname{width}}

\newcommand{\double}{\mathsf{D}}
\newcommand{\D}{\mathrm{d}}

\newcommand{\scal}{{\rm scal}}
\newcommand{\rmc}{\mathrm{c}}

\newcommand{\rmS}{\mathrm{S}}
\newcommand{\rmT}{\mathrm{T}}

\newcommand{\upper}{\uppercase\expandafter}
\newcommand{\romnu}{\romannumeral}

\newcommand{\p}{\partial}
\newcommand{\pM}{{\p M}}

\newcommand{\supp}{\operatorname{supp}}

\newcommand{\coker}{\operatorname{coker}}
\newcommand{\ch}{\operatorname{ch}}
\newcommand{\Tch}{\operatorname{Tch}}

\newcommand{\End}{\operatorname{End}}

\newcommand{\rk}{\operatorname{rk}}
\newcommand{\dist}{\operatorname{dist}}

\newcommand{\cut}{{\rm cut}}
\newcommand{\Kcw}{K\mbox{-cw}_2}
\newcommand{\hAcw}{\hat{A}\mbox{-cw}_2}
\newcommand{\ocw}{\omega\mbox{-cw}_2}
\newcommand{\tr}{\operatorname{tr}}
\newcommand{\Tr}{\operatorname{Tr}}

\newcommand{\spf}{\operatorname{sf}}

\newcommand{\slaD}{\slashed{D}}
\newcommand{\slaS}{\slashed{S}}
\newcommand{\tilD}{\tilde D}
\newcommand{\tilE}{\tilde E}
\newcommand{\tilF}{\tilde F}

\newcommand{\tilS}{\tilde S}

\newcommand{\uD}{\underline{D}}
\newcommand{\uS}{\underline{S}}

\newcommand{\loc}{{\rm loc}}

\begin{document}

\normalsize

\title[Spectral flow, odd K-cowaist, and PSC]{Spectral flow of Callias operators, odd K-cowaist, and positive scalar curvature}

\author[Pengshuai Shi]{Pengshuai Shi${}^\ast$}
\address{School of Mathematics and Statistics, Beijing Institute of Technology, Beijing 100081, P. R. China}

\email{pengshuai.shi@gmail.com, shipengshuai@bit.edu.cn}



\thanks{${}^\ast$Partially supported by the NSFC (grant no. 12101042) and Beijing Institute of Technology Research Fund Program for Young Scholars.}

\begin{abstract}
On a complete Riemannian manifold $M$, we study the spectral flow of a family of Callias operators. We derive a codimension zero formula when the dimension of $M$ is odd and a codimension one formula when the dimension of $M$ is even. These can be seen as analogues of Gromov--Lawson's relative index theorem and classical Callias index theorem, respectively. Secondly, we introduce an intrinsic definition of K-cowaist on odd-dimensional manifolds, making use of the odd Chern character of a smooth map from the manifold to a unitary group. It behaves just like the usual K-cowaist on even-dimensional manifolds. We then apply the notion of odd K-cowaist and the tool of spectral flow to investigate problems related to positive scalar curvature on spin manifolds. In particular, we prove infinite odd K-cowaist to be an obstruction to the existence of PSC metrics. We obtain quantitative scalar curvature estimates on complete non-compact manifolds and scalar-mean curvature estimates on compact manifolds with boundary. They extend several previous results optimally, which unfolds a major advantage of our method via spectral flow and odd K-cowaist.
\end{abstract}

\maketitle


\section{Introduction}\label{S:intro}

Let $D$ be a Dirac operator on a closed Riemannian manifold $M$. Its Fredholm index (or index for short), which is the dimension of $\ker D$ minus the dimension of $\coker D$, is computed topologically by the Atiyah--Singer index theorem. This index is non-zero only when $\dim M$ is even. When $M$ is a complete \emph{non-compact} Riemannian manifold, $D$ can still be a Fredholm operator if it is assumed to be invertible at infinity. In this case, Gromov and Lawson in \cite{GromovLawson83} obtained a \emph{relative index theorem}. For two Dirac operators $D_1$ and $D_2$, if they coincide and are invertible at infinity, then the difference of the indices of $D_1$ and $D_2$ can be transferred to a topological index on a closed manifold derived from the compact part of $M$ where $D_1$ and $D_2$ differ. From this point of view, we have a codimension zero index formula for Dirac operators on an even-dimensional complete Riemannian manifold.

On an odd-dimensional complete non-compact Riemannian manifold $M$, there is another index formula for the so-called Callias operator. It is a Dirac-type operator with potential (see Definition~\ref{D:Callias}). The \emph{Callias index theorem} \cite{Callias78,Anghel93Callias,Bunke95} computes the index of a Callias operator as the index of the Dirac operator on a closed hypersurface of $M$ determined by the potential. This can be seen as a codimension one index formula.

While the index of a Dirac-type operator always vanishes on odd-dimensional closed manifolds, one can consider its counterpart, the spectral flow, in odd dimensions. The spectral flow is an integer valued quantity defined for a continuous family of Fredholm operators (see Subsection~\ref{SS:spf boundary}). As the first topic of this paper, we present codimension zero and codimension one spectral flow formulas for Callias operators on manifolds of opposite dimension parity compared to the above-mentioned index formulas.

\medskip

Let $M$ be a complete Riemannian manifold endowed with a Dirac bundle $S$, and let $\calD_\Phi:C^\infty(M,S)\to C^\infty(M,S)$ be a Callias operator constructed from a Dirac operator $D$ and a bundle endomorphism $\Phi\in\End(S)$ (called Callias potential). Let $\rho\in C^\infty(M,U(l))$ be a smooth function with values in the unitary group $U(l)$ which satisfies certain condition (Assumption~ \ref{A:rho rel bounded}). Then one can define the spectral flow of the family
\[
\calD_\Phi(r):=(1-r)\calD_\Phi+r\rho^{-1}\calD_\Phi\rho,\qquad r\in[0,1].
\]
We denote it by $\spf(\calD_\Phi,\rho)$.

In the odd-dimensional case, there is a way to construct a Callias operator using Gromov--Lawson pair (cf. \cite{Cecchini20LN,CeccZeid24GT}). We consider $S$ to be a twisted bundle $S=\uS\otimes(E\oplus F)$, where $\uS$ is a Dirac bundle and $(E,F)$ is a Gromov--Lawson pair (see Assumption~\ref{A:GL pair}). For a Callias operator $\calD_\Phi$, its \emph{Callias support} is a compact subset $M'\Subset M$ such that $\calD_\Phi^2$ is invertible outside $M'$. As in \cite{CeccZeid24GT}, we can form a vector bundle $V(E,F)$ over $\double M'$ (the double of $M'$) from a Gromov--Lawson pair $(E,F)$ so that it coincides with $E$ on one copy of $M'$ and with $F$ on the other copy. Let $\rho=\rho^+\oplus\rho^-\in C^\infty(M,U(l)\oplus U(l))$ such that $\rho^+=\rho^-$ is locally constant outside a Callias support $M'$ of $\calD_\Phi$. Then one can glue $\rho^+_{|M'}$ and $\rho^-_{|M'}$ to get $\tilde{\rho}\in C^\infty(\double M',U(l))$. Our codimension zero spectral flow formula now can be formulated as the following.

\begin{theorem}[cf. Theorem~\ref{T:codim 0 spf}]\label{IntroT:codim 0 spf}
Let $\calD_\Phi$ be a Callias operator on $S=\uS\otimes(E\oplus F)$ constructed from a Gromov--Lawson pair $(E,F)$ as above over an odd-dimensional complete Riemannian manifold $M$ without boundary. Suppose the Callias potential is a positive (or negative) operator outside a Callias support $M'$. Then for any $\rho=\rho^+\oplus\rho^-$ such that $\rho^+=\rho^-$ is locally constant outside $M'$, we have
\[
\spf(\calD_\Phi,\rho)=\spf(D_{V(E,F)},\tilde{\rho}),
\]
where $D_{V(E,F)}$ is the twisted Dirac operator acting on sections of $\uS\otimes V(E,F)$ over the closed manifold $\double M'$.
\end{theorem}

The way to compute a spectral flow on a complete non-compact manifold is to partition the manifold into a compact part and a non-compact part and use the splitting formula \cite[Theorem~2.10]{Shi24LN} for the spectral flow. In this setting, we consider a self-adjoint local boundary condition defined by an admissible boundary involution $\bfs\in\End(S_{|\pM})$ (see Definition~\ref{D:bdry involution}). For our case, the spectral flow on the non-compact part vanishes, so the question is converted to the spectral flow on a compact manifold with boundary, which corresponds to the following formula.

\begin{corollary}\label{IntroC:codim 0 spf}
Let $M$ be an odd-dimensional compact Riemannian manifold with boundary and $S=\uS\otimes(E\oplus F)$ be a twisted Dirac bundle over $M$ constructed from a Gromov--Lawson pair $(E,F)$. Suppose $\rho=\rho^+\oplus\rho^-\in C^\infty(M,U(l)\oplus U(l))$ satisfies that $\rho^+=\rho^-$ is locally constant near $\pM$. Then for any Callias potential $\Phi\in\End(S)$,
\[
\spf(\calD_{\Phi,\id},\rho)=\spf(D_{V(E,F)},\tilde{\rho}),
\]
where $\calD_{\Phi,\id}$ denotes the Callias operator with boundary condition defined by $\bfs=\id$, $D_{V(E,F)}$ is the twisted Dirac operator on $\double M$, and $\tilde{\rho}\in C^\infty(\double M,U(l))$ is the gluing of $\rho^+$ and $\rho^-$.
\end{corollary}

In the even-dimensional case, we consider the bundle $S=\uS\oplus\uS$ to be two copies of a Dirac bundle $\uS$ over $M$. Then we can get a Callias operator $\calD_\Phi$ acting on sections of $S$ where the potential $\Phi$ is constructed from a bundle endomorphism $\phi\in\End(\uS)$ (see Subsection~\ref{SS:codim 1}). In this setting, we have the following codimension one spectral flow formula.

\begin{theorem}[cf. Theorem~\ref{T:codim 1 spf}]\label{IntroT:codim 1 spf}
Let $\calD_\Phi$ be a Callias operator on $S=\uS\oplus\uS$ over an even-dimensional complete Riemannian manifold $M$ without boundary. Then for any $\rho\in C^\infty(M,U(l))$ satisfying Assumption~\ref{A:rho rel bounded}, we have
\[
\spf(\calD_\Phi,\rho)=\spf(D_{\Sigma+},\rho_{|\Sigma})=-\spf(D_{\Sigma-},\rho_{|\Sigma}),
\]
where $\Sigma$ is a closed hypersurface which is the boundary of a Callias support of $\calD_\Phi$, and $D_{\Sigma\pm}$ are the Dirac operators on $\Sigma$ acting on sections of the subbundles of $\uS_{|\Sigma}$ corresponding to positive/negative eigenspaces of $\phi_{|\Sigma}$.
\end{theorem}

The corresponding formula on manifolds with boundary is phrased for a compact Riemannian \emph{band}, which is a Riemannian manifold $M$ with a decomposition $\pM=\p_+M\sqcup\p_-M$, where $\p_\pm M$ are unions of connected components (cf. \cite{Gromov18metric,CeccZeid24GT}). In this case, the following holds.

\begin{corollary}\label{IntroC:codim 1 spf}
Let $M$ be an even-dimensional compact Riemannian band endowed with a Dirac bundle $S=\uS\oplus\uS$. Choose the admissible boundary involution $\bfs=\pm\id\in\End(S_{|\pM})$ on $\p_\pm M$. Then for any Callias potential $\Phi\in\End(S)$ and any $\rho\in C^\infty(M,U(l))$,
\[
\spf(\calD_{\Phi,\bfs},\rho)=\spf(D_{\p_+M},\rho_{|\p_+M})=-\spf(D_{\p_-M},\rho_{|\p_-M}).
\]
\end{corollary}

\begin{remark}\label{IntroR:spf for}
Corollaries~\ref{IntroC:codim 0 spf} and \ref{IntroC:codim 1 spf} are spectral flow analogues of \cite[Corollaries~3.9 and 3.10]{CeccZeid24GT}. Theorem~\ref{IntroT:codim 1 spf} is the spectral flow analogue of the classical Callias index theorem. Recently, van den Dungen \cite{Dungen25,Dungen24} developed a systematic study of the index theory of Callias operators in an abstract manner.
\end{remark}

\medskip

The second topic of this paper is to have a discussion of the notion of K-cowaist in odd dimensions. The K-cowaist of an oriented Riemannian manifold was first proposed and studied by Gromov in \cite{Gromov96} under the name of K-area. The \emph{K-cowaist} of a closed oriented Riemannian manifold $M$, denoted by $\Kcw(M)$, is defined to be $(\inf_E\|R^E\|)^{-1}$ with $E$ ranging over all Hermitian vector bundles over $M$ with Hermitian connections such that $E$ has a non-trivial Chern number (see Subsection~\ref{SS:K-cw even}). 
Clearly, this is a notion initially making sense only for even-dimensional manifolds. Although there have been several ways of generalizing K-cowaist to odd-dimensional manifolds in the literature, they are all based on the idea of transferring the space to even dimension by taking product with $\rmS^1$ or $\RR$. In this paper, we attempt to give a more intrinsic definition of K-cowaist on odd-dimensional manifolds.

Basically, on an odd-dimensional oriented Riemannian manifold $M$, we will consider a pair $(E;\rho)$, where $E\to M$ is a Hermitian vector bundle with a Hermitian connection and $\rho$ is a smooth map from $M$ to a unitary group $U(l)$. If the manifold is non-compact or with boundary, $(E;\rho)$ should satisfy certain compatibility condition at infinity or near the boundary (see Subsection~\ref{SS:K-cw odd}). Now instead of requiring $E$ to have a non-trivial Chern number, we require $(E;\rho)$ to have a non-trivial \emph{odd} Chern number, which is constructed by integrating over $M$ the wedge product of Chern character forms of $E$ and an odd Chern character form of $\rho$. Such pairs are called admissible pairs. Then we can define the K-cowaist as follows.

\begin{definition}\label{IntroD:K-cw odd}
Let $M$ be an odd-dimensional oriented Riemannian manifold. The \emph{K-cowaist} of $M$ is defined to be
\[
\Kcw(M):=\left(\inf\{\|R^E\|+\|R^\rho\|\}\right)^{-1}
\]
with $(E;\rho)$ ranging over all admissible pairs, where the curvature $R^\rho$ is defined in \eqref{E:curv rho}.
\end{definition}

Given a mixed differential form $\omega$ on $M$ of even degree, we can similarly define the $\omega$-cowaist of $M$ (cf. Definition~\ref{D:K-cw odd}), denoted by $\ocw(M)$. As a special case, we can talk about the $\hat{A}$-cowaist of an odd-dimensional manifold. The K-cowaist can be bounded by the $\omega$-cowaist.

\begin{theorem}\label{IntroT:K-cw omega-cw}
There exists a constant $c(n)$ depending only on the dimension $n$ of $M$ such that $\Kcw(M)\le c(n)\cdot\ocw(M)$.
\end{theorem}

This relationship was established originally in even dimensions in \cite{Gromov96,Wang23K-cw} for the case that $\omega$ is the $\hat{A}$-genus form of $M$ and in \cite{BaerHanke24K-cw} for the general case. Note that the converse inequality does not hold (cf. \cite{Wang23K-cw}). A key ingredient in the proof of Theorem~\ref{IntroT:K-cw omega-cw} is the Adams operation on vector bundles. In order to prove the result for odd-dimensional manifolds, we need an analogous operation on elements of $C^\infty(M,U(l))$ (see Definition~\ref{D:Adams op odd}).

\medskip

The last part of this paper is concerned with applications of the first two parts to the study of positive scalar curvature (PSC for short) problems on \emph{spin} manifolds. In fact, one of the main motivations of introducing K-cowaist was to give an obstruction to the existence of PSC metrics on a Riemannian spin manifold.

\begin{theorem}[{\cite[Section~5]{Gromov96}}, \cite{BaerHanke24K-cw}]\label{IntroT:K-cw psc even}
Let $(M,g)$ be a $2m$-dimensional complete Riemannian spin manifold with (possibly empty) compact mean convex boundary. If the scalar curvature $\scal_g\ge\sigma^2\cdot2m(2m-1)$ for some $\sigma>0$, then $\hAcw(M)\le2\sigma^{-2}$, and therefore $\Kcw(M)\le C(m)\sigma^{-2}$ for some constant $C(m)$ depending only on $m$.

In particular, if $M$ has infinite K-cowaist, then it cannot admit a complete Riemannian metric of uniformly positive scalar curvature and non-negative mean curvature along the boundary.
\end{theorem}

This theorem is proved by using the relative index theorem of spin Dirac operators combined with Theorem~ \ref{IntroT:K-cw omega-cw}. Now that we have established the codimension zero spectral flow formula, together with our definition of odd K-cowaist, we are able to generalize Theorem~\ref{IntroT:K-cw psc even} to odd dimensions.

\begin{theorem}[cf. Theorem~\ref{T:K-cw psc odd}]\label{IntroT:K-cw psc odd}
Theorem~\ref{IntroT:K-cw psc even} holds for odd-dimensional Riemannian spin manifolds as well.
\end{theorem}

A typical class of manifolds with infinite K-cowaist is the even-dimensional finitely area enlargeable manifolds. 

\begin{definition}\label{IntroD:area enlarg}
An oriented Riemannian $n$-manifold $M$ is said to be \emph{area enlargeable} if given any Riemannian metric on $M$ and any $\varepsilon>0$, there exist a covering manifold $\widetilde{M}\to M$ (with the lifted metric) and a smooth map $f:\widetilde{M}\to\rmS^n(1)$ (the standard unit sphere) of non-zero degree which is locally constant at infinity and near the boundary such that $f$ is $\varepsilon$-area-contracting, which means that $|f^*\alpha|\le\varepsilon|\alpha|$ for any two form $\alpha$ on $\rmS^n(1)$.

If in addition, $\widetilde{M}\to M$ is a finite covering, we call $M$ a \emph{finitely} area enlargeable manifold.
\end{definition}

\begin{remark}\label{R:area enlarge}
The fact that a (not necessarily finitely) area enlargeable manifold without boundary does not admit a complete PSC metric has been proved by Gromov--Lawson \cite{GromovLawson83} using their relative index theorem. Compared with their original definition, here we relax the condition for $f$ to be \emph{locally} constant instead of constant at infinity and near the boundary. This will make a difference in the argument for odd-dimensional case, as it will not be easy to carry out the usual trick of taking product with $\rmS^1$ so to get an even-dimensional area enlargeable manifold.
\end{remark}

In Proposition~\ref{P:area enlarg K-cw odd}, we show that an odd-dimensional finitely area enlargeable manifold also has infinite (odd) K-cowaist. Therefore by Theorem~\ref{IntroT:K-cw psc odd}, it cannot admit a metric of uniformly PSC and non-negative mean curvature along the boundary.

Using similar ideas as proving Theorem~\ref{IntroT:K-cw psc odd}, we further deduce several results interacting odd K-cowaist or $\hat{A}$-cowaist with PSC problems. In the case of complete spin manifolds without boundary, we prove in Theorem~\ref{T:quant est codim 0} a codimension zero quantitative scalar curvature estimate, which in particular implies a quantitative Llarull's theorem on non-compact spin manifolds (Corollary~ \ref{C:quan est codim 0-contr}) obtained in \cite{Shi24LN}. In the case of compact spin manifolds with boundary, we prove a codimension zero scalar-mean curvature estimate in Theorem~\ref{T:scal-mean est codim 0} for the general long neck problem and in Theorem~\ref{T:scal-mean est codim 0-inf} for the problem about the width of geodesic collar neighborhood of the boundary.

Recall that the \emph{width} of a compact Riemannian band $(M,g)$ with $\pM=\p_+M\sqcup\p_-M$ is defined as
\[
\wid(M,g):=\dist_g(\p_+M,\p_-M).
\]
If $M$ is diffeomorphic to $N\times[-1,1]$ for some manifold $N$, then $M$ is a trivial band. In this case, $\p_\pm M$ are chosen to be the two copies of $N$, and we deduce a codimension one scalar-mean curvature estimate for the band width problem.

\begin{theorem}\label{IntroT:scal-mean est codim 1}
Let $(M,g)$ be an $n$-dimensional ($n\ge2$ even) Riemannian spin band which is diffeomorphic to $N\times[-1,1]$, where $N$ is a closed manifold of infinite K-cowaist. Suppose $\scal_g\ge\sigma^2n(n-1)$ on $M$ for some $\sigma>0$, and there exist $-\frac{\pi}{\sigma n}<l_-<l_+<\frac{\pi}{\sigma n}$ such that the mean curvature of $\pM$ satisfies
\[
H_{g|\p_\pm M}\ge\mp\tan\Big(\frac{1}{2}\sigma nl_\pm\Big).
\]
Then $\wid(M,g)\le l_+-l_-$.
\end{theorem}

\begin{remark}\label{IntroR:scal-mean est codim 1}
Theorem~\ref{IntroT:scal-mean est codim 1} generalizes the scalar-mean curvature estimate for the band width problem in \cite{CeccZeid24GT} to the case of odd K-cowaist. It applies to examples such as torical bands $\rmT^{n-1}\times[-1,1]$, overtorical bands (cf. \cite{Gromov18metric}) and more generally $\hat{A}$-overtorical bands (cf. \cite{Zeidler20}). Theorem~\ref{IntroT:scal-mean est codim 1} also implies a result (Corollary~\ref{C:non-exist NxR}) related to the ``$\times\RR$ stability'' conjecture of Rosenberg and Stolz.
\end{remark}

Theorem~\ref{IntroT:scal-mean est codim 1}, together with the other quantitative results in Section~ \ref{S:odd K-cw psc}, mostly extends the previous results to the other dimension parity. However, note that the parameters in the hypothesis and conclusion of these theorems depend on the dimension $n$ of the manifold. The usual way of taking product with $\rmS^1$ and converting to the other dimension parity will not lead to quantitatively optimal estimates, as was pointed out in \cite[Remark~4.20]{CeccZeid23}. Our new method involving spectral flow and odd K-cowaist overcomes this issue and sets up the results in the optimal form. We believe this method can potentially be applied to more situations.

As a general remark, note that both infinite K-cowaist and non-vanishing Rosenberg index are obstructions to the existence of PSC metrics on a spin manifold. It may be a question worth considering whether the latter can be inferred by the former (see Remark~\ref{R:band width & non-exist NxR}).

\subsection*{Organization of this paper}\label{SS:orga}

The paper is organized as follows. In Section~\ref{S:spf Callias}, we study the spectral flow for a family of Callias operators and prove the codimension zero and codimension one formulas. In Section~\ref{S:K-cw odd}, we first recall the notion of K-cowaist in even dimensions, then give an intrinsic definition for odd-dimensional manifolds. In Section~\ref{S:odd K-cw psc}, we use the tool of spectral flow and the notion of odd K-cowaist to prove the aforementioned results related to PSC problems.

\subsection*{Acknowledgments}\label{SS:acknow}

The author wishes to thank the referee for helpful comments and suggestions which improved the presentation of the paper.

\section{Spectral flow formulas for Callias operators}\label{S:spf Callias}

\subsection{Callias operators and boundary chirality}\label{SS:Callias}

A Callias operator is roughly a Dirac-type operator whose square is a Schr\"odinger operator such that the potential has a positive lower bound at infinity. Let $M$ be an $n$-dimensional complete Riemannian manifold (possibly with compact boundary), and $S\to M$ be a Hermitian vector bundle with a Hermitian connection $\nabla$. $S$ is called a \emph{Dirac bundle} if it is endowed with a Clifford multiplication $\rmc(\cdot):T^*M\to\End(S)$ which is skew-adjoint and satisfies $\rmc(\cdot)^2=-|\cdot|^2$, and is compatible with $\nabla$ (i.e. $\rmc(\cdot)$ is a parallel bundle map, cf. \cite[\S II.5]{LawMic89}). The \emph{Dirac operator} on a Dirac bundle $S$ is a formally self-adjoint first-order elliptic differential operator defined by
\[
D:=\sum_{i=1}^{n}\rmc(e_i^*)\nabla_{e_i}:\ C^\infty(M,S)\to C^\infty(M,S),
\]
where $e_1,\dots,e_n$ is an orthonormal local tangent frame and $e_1^*,\dots,e_n^*$ is the associated dual cotangent frame. 

Let $S_{|\pM}$ be the restriction of $S$ to the boundary. Then $S_{|\pM}$ is also a Dirac bundle by setting
\[
\begin{aligned}
\rmc^\p(\xi)&:=\rmc(\nu^*)^{-1}\rmc(\xi), \\
\nabla^\p&:=\nabla+\frac{1}{2}\rmc^\p(\nabla\nu^*),
\end{aligned}
\]
where $\xi\in T^*\p M$ and $\nu^*$ is the dual covector of the inward pointing unit normal vector field $\nu=e_n$ along $\p M$. The corresponding Dirac operator
\[
A:=\sum_{i=1}^{n-1}\rmc^\p(e_i^*)\nabla^\p_{e_i}:\ C^\infty(\pM,S_{|\pM})\to C^\infty(\pM,S_{|\pM})
\]
anti-commutes with $\rmc(\nu^*)$ and is called a \emph{compatible adapted operator} to $D$.  Also, we have
\begin{equation}\label{E:bdry Dirac mean}
A=\rmc(\nu^*)^{-1}D-\nabla_\nu+\frac{n-1}{2}H,
\end{equation}
where $H$ is the mean curvature along $\pM$ with respect to $\nu$, cf. \cite[Appendix~1]{BaerBallmann16}.

A convenient way to talk about a Callias operator is using the notion of relative Dirac bundle of Cecchini--Zeidler \cite{CeccZeid24GT}.

\begin{definition}[{\cite[Definition~2.2]{CeccZeid24GT}}]\label{D:rel Dirac}
Let $K\Subset M^\circ$ be a compact subset in the interior of $M$. A Dirac bundle $S\to M$ is said to be a \emph{relative Dirac bundle} with support $K$ if there is a self-adjoint, parallel bundle involution $\theta\in C^\infty(M\setminus K,\End(S))$ such that $\rmc(\xi)\theta=-\theta\rmc(\xi)$ for any $\xi\in T^*M_{|M\setminus K}$ and $\theta$ admits a smooth extension to a bundle endomorphism on an open neighborhood of $\overline{M\setminus K}$.
\end{definition}

Let $(S,\theta)$ be a relative Dirac bundle with support $K\Subset M^\circ$. Consider a self-adjoint bundle endomorphism $\Phi\in\End(S)$ which commutes with $\theta$ such that $\Phi=0$ on $K$. Extending $\Phi\theta$ by zero on $K$, one can construct a formally self-adjoint Dirac-type operator
\begin{equation}\label{E:Callias}
\calD_\Phi:=D+\Phi\theta,
\end{equation}
where $D$ is the Dirac operator on $S$.

\begin{definition}\label{D:Callias}
We call the operator $\calD_\Phi=D+\Phi\theta$ a \emph{Callias operator} if the commutator $[D,\Phi]$ is a bundle endomorphism\footnote{This means that $\Phi$ commutes with Clifford multiplication.} 
and that $\Phi^2-|[D,\Phi]|$ is a uniformly positive bundle endomorphism outside a compact subset $K'\Subset M^\circ$, where $|[D,\Phi]|(p)$ denotes the operator norm of $[D,\Phi](p):S_p\to S_p$ for any $p\in M$.

In this case, $\Phi$ is called a \emph{Callias potential} and $K'$ is called a \emph{Callias support}.
\end{definition}

\begin{remark}\label{R:Callias potential}
Note that a Callias support  always contains the support of the relative Dirac bundle. Our Callias potential $\Phi$ is usually assumed to be smooth. But as in \cite{CeccZeid24GT,Shi24LN}, we also allow potentials which are Lipschitz continuous on a compact subset. See Remark~\ref{R:Lip potential}.
\end{remark}

From this definition, a Callias operator $\calD_\Phi$ is formally self-adjoint, and
\begin{equation}\label{E:Callias square}
\calD_\Phi^2=D^2+\Phi^2+[D,\Phi]\theta,
\end{equation}
which means that it is invertible (or coercive) at infinity.

Suppose $\pM$ is non-empty. We focus on a special kind of local boundary conditions for a Callias operator as follows.

\begin{definition}\label{D:bdry involution}
A bundle endomorphism $\bfs\in\End(S_{|\pM})$ is said to be an \emph{admissible boundary involution} if it is a self-adjoint bundle involution which commutes with the Clifford multiplication. The \emph{boundary chirality} associated to $\bfs$ is defined to be
\[
\chi:=\bfs\rmc(\nu^*)\theta:S_{|\pM}\to S_{|\pM}.
\]
\end{definition}

Note that $\chi$ is also a self-adjoint involution. Thus it induces an orthogonal decomposition $S_{|\pM}=S^+_\pM\oplus S^-_\pM$, where $S^\pm_\pM$ are the $\pm1$-eigenspaces of $\chi$. Since $\chi$ anti-commutes with $\rmc(\nu^*)$ while commutes with $\rmc(\xi)$ for $\xi\in T^*\p M$, it anti-commutes with $\rmc(\nu^*)^{-1}\rmc(\xi)$. This means that $\rmc(\nu^*)$ interchanges $S^+_\pM$ and $S^-_\pM$, and that $\chi$ anti-commutes with $A$ (a compatible adapted operator to $D$). By \cite[Example~4.20]{BaerBallmann16}, $\chi$ induces a self-adjoint elliptic local boundary condition for $\calD_\Phi$ with the domain given by
\[
\dom(\calD_{\Phi,\bfs}):=\left\{u\in H^1_{\calD_\Phi}(M,S)\;|\;\chi(u_{|\p M})=u_{|\p M}\right\},
\]
where
\[
H^1_{\calD_\Phi}(M,S):=\{u\in H^1_\loc(M,S)\cap L^2(M,S)\;|\;\calD_\Phi u\in L^2(M,S)\}.
\]
Therefore, $\calD_{\Phi,\bfs}$ is a self-adjoint Fredholm operator.

\subsection{Spectral flow on manifolds with boundary}\label{SS:spf boundary}

In this subsection, we study the spectral flow of a family of Callias operators on manifolds with boundary and give a splitting formula for the spectral flow. This will be used to derive more concrete formulas in later subsections.

Let $M$ be a complete Riemannian manifold with compact boundary and $\calD_\Phi=D+\Phi\theta$ be a Callias operator on $M$. Impose to $\calD_\Phi$ the local boundary condition determined by an admissible boundary involution $\bfs\in\End(S_{|\pM})$ as in last subsection. Let $\rho\in C^\infty(M,U(l))$ be a smooth function on $M$ with values in the unitary group $U(l)$, which satisfies the following.

\begin{assumption}\label{A:rho rel bounded}
The commutator $[D,\rho]=\rmc(\D\rho)$ is bounded as an operator from $\dom(\calD_{\Phi,\bfs})$ to $L^2(M,S)$.
\end{assumption}

Then $\rho$ preserves $\dom(\calD_{\Phi,\bfs})$. Note that $\rho$ commutes with both $\Phi$ and $\theta$. For $r\in[0,1]$, define a family of operators\footnote{Here we identify $\calD_\Phi$ and $\rho$ with their natural extensions acting on the bundle $S\otimes\CC^l\to M$. See Remark~\ref{R:family conn}.}
\begin{equation}\label{E:Callias family}
\calD_\Phi(r):=(1-r)\calD_\Phi+r\rho^{-1}\calD_\Phi\rho=D(r)+\Phi\theta
\end{equation}
on the domain $\dom(\calD_{\Phi,\bfs})$, where $D(r)=(1-r)D+r\rho^{-1}D\rho=D+r\rho^{-1}[D,\rho]$. It can be checked that $\calD_\Phi(r)$ is a family of Callias operators whose Callias support is independent of $r$. Moreover, under Assumption~\ref{A:rho rel bounded}, $\calD_{\Phi,\bfs}(r)$ is a continuous family of self-adjoint Fredholm operators in the Riesz topology (cf. \cite{BoossLeschPhillips05}, \cite[Proposition~ 2.2]{Lesch05sf}). So one can consider the \emph{spectral flow} of this family, denoted by $\spf(\calD_{\Phi,\bfs},\rho)$, which is the net number of eigenvalues of $\calD_{\Phi,\bfs}(r)$ that change from negative to non-negative as $r$ varies from 0 to 1. (In particular, we adopt the convention of \cite{GorokhovskyLesch15} for the degenerate case where endpoint values may not be invertible.)

\begin{remark}\label{R:family conn}
It is worth pointing out that $D(r)$ can be regarded as a Dirac operator on a new Dirac bundle with varying connections. Precisely, consider the twisted bundle $S\otimes\CC^l$. It is again a Dirac bundle with the Clifford multiplication acting as identity on $\CC^l$ and with a family of connections given by
\[
\nabla^{S\otimes\CC^l}(r):=\nabla^S\otimes\id+\id\otimes(\D+r\rho^{-1}[\D,\rho]),\quad r\in[0,1].
\]
The corresponding Dirac operator is just $D(r)=(1-r)D+r\rho^{-1}D\rho$, where $D$ is the Dirac operator on $S$ (identified with $D(0)$). The maps $\theta$ and $\Phi$ can also be extended to acting on $S\otimes\CC^l$. But for simplicity, usually we will not distinguish between $S$ and $S\otimes\CC^l$.
\end{remark}

Like Fredholm index, the spectral flow is also stable up to compact perturbations

\begin{lemma}\label{L:spf homotop inv}
Let $\Phi_1,\Phi_2$ be two Callias potentials on $M$ which coincide outside a compact subset. Then for any $\rho\in C^\infty(M,U(l))$ satisfying Assumption~\ref{A:rho rel bounded} and any admissible boundary involution $\bfs\in\End(S_{|\pM})$,
\[
\spf(\calD_{\Phi_1,\bfs},\rho)=\spf(\calD_{\Phi_2,\bfs},\rho).
\]
\end{lemma}

\begin{proof}
Since $\eta:=\calD_{\Phi_2}-\calD_{\Phi_1}=\Phi_2\theta-\Phi_1\theta$ is a bundle endomorphism with compact support, $u\in H^1_{\calD_{\Phi_1}}(M,S)$ if and only if $u\in H^1_{\calD_{\Phi_2}}(M,S)$. Combined with the fact that the boundary chirality used to define the local boundary condition is independent of the Callias potential, one concludes that $\dom(\calD_{\Phi_1,\bfs})=\dom(\calD_{\Phi_2,\bfs})$.

Note that $[\calD_{\Phi_1},\eta]$ is also a bundle endomorphism with compact support, thus having finite $L^\infty$-norm. For $u\in\dom(\calD_{\Phi_1,\bfs})$,
\[
\|\calD_{\Phi_1}\eta u\|_{L^2}\le\|\eta\calD_{\Phi_1}u\|_{L^2}+\|[\calD_{\Phi_1},\eta]u\|_{L^2}\le\|\eta\|_{L^\infty}\|\calD_{\Phi_1}u\|_{L^2}+\|[\calD_{\Phi_1},\eta]\|_{L^\infty}\|u\|_{L^2}.
\]
So $\calD_{\Phi_2}-\calD_{\Phi_1}$ is a bounded operator from $\dom(\calD_{\Phi_1,\bfs})$ to $H^1_{\calD_{\Phi_1}}(M,S)$. In fact, it is a bounded operator from $\dom(\calD_{\Phi_1,\bfs})$ to $H^1(\supp(\eta),S)$. By Rellich lemma, $H^1(\supp(\eta),S)\hookrightarrow L^2(M,S)$ is a compact embedding.  It then follows that $\calD_{\Phi_2}-\calD_{\Phi_1}:\dom(\calD_{\Phi_1,\bfs})\to L^2(M,S)$ is a compact operator.

For $0\le r\le1$, we proceed as in the proof of \cite[Proposition~2.1]{GorokhovskyLesch15} to consider the following concatenation of paths
\[
\begin{aligned}
f_1(r)=\calD_{(1-r)\Phi_1+r\Phi_2,\bfs},\qquad f_2(r)=(1-r)\calD_{\Phi_2,\bfs}+r\rho^{-1}\calD_{\Phi_2,\bfs}\rho, \\
f_3(r)=\rho^{-1}\calD_{(1-r)\Phi_2+r\Phi_1,\bfs}\rho,\qquad f_4(r)=(1-r)\rho^{-1}\calD_{\Phi_1,\bfs}\rho+r\calD_{\Phi_1,\bfs},
\end{aligned}
\]
which is homotopic to the constant path $\calD_{\Phi_1,\bfs}$. Each path is a continuous family of self-adjoint Fredholm operators acting on the fixed domain $\dom(\calD_{\Phi_1,\bfs})$. By the homotopy invariance of the spectral flow \cite[Proposition~2.3]{BoossLeschPhillips05},
\[
\spf(f_1)+\spf(f_2)+\spf(f_3)+\spf(f_4)=0.
\]
The lemma is now proved by observing that $\spf(f_1)+\spf(f_3)=0$, $\spf(f_2)=\spf(\calD_{\Phi_2,\bfs},\rho)$, and $\spf(f_4)=-\spf(\calD_{\Phi_1,\bfs},\rho)$.
\end{proof}

Before further discussion of the spectral flow, we divert to some computations that will induce spectral estimates for Callias operators and will be used frequently in this paper. Consider the Callias operators $\calD_{\Phi,\bfs}(r)$, $0\le r\le1$, as in \eqref{E:Callias family} with boundary condition given by an admissible boundary involution $\bfs\in\End(S_{|\pM})$. By combining \eqref{E:bdry Dirac mean}, \eqref{E:Callias square}, Weitzenb\"ock formula for Dirac operator, and Green's formula, we have for any $u\in\dom(\calD_{\Phi,\bfs}(r))=\dom(\calD_{\Phi,\bfs}(0))$,
\begin{equation}\label{E:Callias square int-1}
\begin{aligned}
\int_M|\calD_\Phi(r)u|^2 \D V_M& =\int_M\left(|D(r)u|^2+\langle u,(\Phi^2+[D(r),\Phi]\theta)u\rangle\right)\D V_M \\
&\quad+\int_\pM\langle u,\bfs\Phi u\rangle\D V_\pM \\
& =\int_M|\nabla(r)u|^2\D V_M+\int_M\langle u,(\calR(r)+\Phi^2+[D(r),\Phi]\theta)u\rangle\D V_M \\
&\quad+\int_\pM\left\langle u,\left(\bfs\Phi+\frac{n-1}{2}H-A(r)\right)u\right\rangle\D V_\pM \\
& =\int_M|\nabla(r)u|^2\D V_M+\int_M\langle u,(\calR(r)+\Phi^2+[D(r),\Phi]\theta)u\rangle\D V_M \\
&\quad+\int_\pM\left\langle u,\left(\bfs\Phi+\frac{n-1}{2}H\right)u\right\rangle\D V_\pM,
\end{aligned}
\end{equation}
where $\calR(r)=\sum_{i<j}\rmc(e_i^*)\rmc(e_j^*)R(r)(e_i,e_j)$ is a curvature endomorphism of the curvature $R(r)$ of $\nabla(r)$, and the last equality is because the boundary chirality defined in Definition~ \ref{D:bdry involution} anti-commutes with $A(r)$.

We then use the identity
\[
|\nabla u|^2=|\calP u|^2+\frac{1}{n}|Du|^2,
\]
where $\calP$ is the \emph{Penrose operator} defined by
\[
\calP_eu:=\nabla_eu+\frac{1}{n}\rmc(e^*)Du,\quad e\in TM.
\]
And \eqref{E:Callias square int-1} can be rearranged as

\begin{equation}\label{E:Callias square int-2}
\begin{aligned}
\int_M|\calD_\Phi(r)u|^2 \D V_M & =\frac{n}{n-1} \int_M(|\calP(r)u|^2+\langle u,\calR(r)u\rangle)\D V_M \\
  &\quad +\int_M\langle u,(\Phi^2+[D(r),\Phi]\theta)u\rangle \D V_M \\
  &\quad +\int_\pM\left\langle u,\left(\bfs\Phi+\frac{1}{2}n H\right)u\right\rangle\D V_{\pM}.
\end{aligned}
\end{equation}

\begin{remark}\label{R:Callias square int}
To be rigorous, the above computations should be performed first to
\[
u\in C^\infty_{\theta,\bfs}:=\left\{u\in C^\infty_c(M,S)\;|\;\chi(u_{|\p M})=u_{|\p M}\right\}.
\]
Then by the density of $C^\infty_{\theta,\bfs}$ in $\dom(\calD_{\Phi,\bfs}(r))$ (cf. \cite{BaerBallmann12}), they hold for $u\in\dom(\calD_{\Phi,\bfs}(r))$.

Using \eqref{E:Callias square int-1} and \eqref{E:Callias square int-2}, one will obtain several different forms of lower bound estimates for $\int_M|\calD_{\Phi,\bfs}(r)u|^2\D V_M$ by dropping some non-negative terms on the right-hand side. This can be done freely when $M$ is a compact manifold. When $M$ is non-compact, one needs to pay attention to the integrability. For example, if the right-hand side can be written as the sum of some a priori non-negative terms, then one still has the estimates. See \cite[Section~ 2]{GromovLawson83}, \cite[Remark~3.8]{Shi24LN}.
\end{remark}

We now turn back to the spectral flow and give the following vanishing result.

\begin{proposition}\label{P:spf vanishing}
Let $M$ be a complete Riemannian manifold with compact boundary, endowed with a relative Dirac bundle $(S,\theta)$. Let $\calD_\Phi=D+\Phi\theta$ be a Callias operator and $\bfs\in\End(S_{|\pM})$ be an admissible boundary involution. Suppose either of the following holds,
\begin{enumerate}
\item $\calD_\Phi$ has empty Callias support and $\bfs\Phi$ is non-negative on $S_{|\pM}$; \label{IT:spf vanishing-1}

\item $M$ is compact, $(S,\theta)$ has empty support, and $\bfs=\id$ (or $\bfs=-\id$) on $S_{|\pM}$. \label{IT:spf vanishing-2}
\end{enumerate}
Then for all $r\in[0,1]$ and $\rho\in C^\infty(M,U(l))$ satisfying Assumption~\ref{A:rho rel bounded}, we have $\spf(\calD_{\Phi,\bfs},\rho)=0$.
\end{proposition}

\begin{proof}
(\romnu1) If $\calD_\Phi$ has empty Callias support, then each $\calD_\Phi(r)$ has empty Callias support. For any $u\in\dom\calD_{\Phi,\bfs}(r)$, by \eqref{E:Callias square int-1},
\[
\begin{aligned}
\int_M|\calD_{\Phi}(r)u|^2\D V_M\ge& \int_M\langle u,(\Phi^2-|[D(r),\Phi]|)u\rangle\D V_M+\int_\pM\langle u,\bfs\Phi u\rangle\D V_\pM \\
\ge& \;C\int_M|u|^2\D V_M
\end{aligned}
\]
for some constant $C>0$. This shows that $\calD_{\Phi,\bfs}(r)$ is invertible for any $r\in[0,1]$. Therefore, the spectral flow vanishes.

(\romnu2) Since $M$ is compact and the relative Dirac bundle has empty support, any self-adjoint bundle endomorphism $\Phi\in\End(S)$ can serve as a Callias potential, provided that $[D,\Phi]$ is a bundle endomorphism. By Lemma~\ref{L:spf homotop inv}, the spectral flow is independent of the Callias potential. So we can choose the special bundle endomorphism $\Phi=\id$ for $\bfs=\id$ (or $\Phi=-\id$ for $\bfs=-\id$). In this case, $\calD_\Phi$ satisfies the hypothesis of part \ref{IT:spf vanishing-1}, which implies vanishing of the spectral flow.
\end{proof}

We now study the spectral flow on partitioned manifolds. For simplicity, we assume that $M$ is without boundary.
Let $\Sigma$ be a closed hypersurface of $M$ lying outside a Callias support of $\calD_\Phi$ with trivial normal bundle. We cut $M$ along $\Sigma$ to get a manifold $M^\cut$ whose boundary consists of two copies $\Sigma'$ and $\Sigma''$ of $\Sigma$. Denote by $\calD_\Phi^\cut:C^\infty(M^\cut,S_{|M^\cut})\to C^\infty(M^\cut,S_{|M^\cut})$ the Callias operator on $M^\cut$ induced by $\calD_\Phi$. Consider the local boundary condition determined by an admissible boundary involution $\bfs\in\End(S_{|\p M^\cut})$.

\begin{theorem}\label{T:splitting spf}
Under the above setting, if $\bfs_{|\Sigma'}=\bfs_{|\Sigma''}$, then for any $\rho\in C^\infty(M,U(l))$ satisfying Assumption~\ref{A:rho rel bounded}, $\spf(\calD_\Phi,\rho)=\spf(\calD^\cut_{\Phi,\bfs},\rho)$.
\end{theorem}

\begin{proof}
Let $Q:=\frac{1}{2}(\id-\chi_{|\Sigma'})$ be the orthogonal projection to $L^2(\Sigma',S_{|\Sigma'})$. Then the boundary condition on $\Sigma'$ is given by $Q(u_{|\pM})=0$. By the hypothesis about $\bfs$, the corresponding projection on $\Sigma''$ is
\[
\frac{1}{2}(\id-\chi_{|\Sigma''})=\frac{1}{2}(\id+\chi_{|\Sigma'})=\id-Q,
\]
since $\rmc(\nu^*_{|\Sigma'})=-\rmc(\nu^*_{|\Sigma''})$. This means that the boundary condition on the whole $\p M^\cut$ is defined by an element of the self-adjoint Fredholm Grassmannian in the sense of \cite[Definition~2.3]{Shi24LN}. The formula then follows from \cite[Theorem~2.10]{Shi24LN}.
\end{proof}

\subsection{A codimension zero spectral flow formula on odd-dimensional manifolds}\label{SS:codim 0}

In this subsection, we take a closer look at the spectral flow of Callias operators on an odd-dimensional manifold. Basically, this has been discussed in \cite[Subsection~3.3]{Shi24LN}. We reformulate the result here with a little more generality.

Let $M$ be an odd-dimensional complete Riemannian manifold without boundary, endowed with a Dirac bundle $\uS\to M$. Let $E,F\to M$ be two Hermitian vector bundles with Hermitian connections. Then the twisted bundle
\begin{equation}\label{E:twist bundle odd}
S:=\uS\otimes(E\oplus F)=(\uS\otimes E)\oplus(\uS\otimes F)
\end{equation}
is a Dirac bundle over $M$, with the connection being the usual tensor product connection and the Clifford multiplication given by
\[
\rmc(\xi)=\left(
\begin{matrix}
\underline{\rmc}(\xi)\otimes\id_E & 0 \\
0 & -\underline{\rmc}(\xi)\otimes\id_F
\end{matrix}\right),\quad\xi\in T^*M,
\]
where $\underline{\rmc}(\xi)$ is the Clifford multiplication on $\uS$. The associated Dirac operator has the form
\begin{equation}\label{E:twist Dirac odd}
D=\left(
\begin{matrix}
D_E & 0 \\
0 & -D_F
\end{matrix}\right),
\end{equation}
where $D_E,D_F$ are the corresponding Dirac operators twisted by $E,F$, respectively. We call the operator $D$ twisted by $E\oplus F^{\rm op}$.

In order for $S$ to have a relative Dirac bundle structure, we make the following assumption as in \cite{Cecchini20LN,CeccZeid24GT} to $E,F$.

\begin{assumption}\label{A:GL pair}
$(E,F)$ is a \emph{Gromov--Lawson pair} with support $K$, that is, there exist a compact subset $K\Subset M$ and a parallel unitary bundle isomorphism $\frakt:E_{|M\setminus K}\to F_{|M\setminus K}$ which extends to a smooth bundle map on a neighborhood of $\overline{M\setminus K}$.
\end{assumption}

Under this assumption, $S$ becomes a relative Dirac bundle with involution
\begin{equation*}\label{E:GL pair rel Dirac}
\theta=\left(
\begin{matrix}
0 & \id\otimes\frakt^* \\
\id\otimes\frakt & 0
\end{matrix}\right):S_{|M\setminus K}\to S_{|M\setminus K}.
\end{equation*}
So we can talk about the Callias operator $\calD_\Phi$ \eqref{E:Callias} in this case.

Let $\Sigma$ be a (not necessarily connected) closed hypersurface of $M$ with trivial normal bundle such that $M$ is partitioned along $\Sigma$ as
\begin{equation}\label{E:partition}
M^\cut=M'\sqcup_\Sigma M'',
\end{equation}
where $M'$ is a Callias support of $\calD_\Phi$ (thus is a compact set containing the set $K$ in Assumption~ \ref{A:GL pair}). The boundaries $\pM'$ and $\pM''$ are two copies of $\Sigma$, denoted by $\Sigma'$ and $\Sigma''$, respectively.
From the Gromov--Lawson pair $(E,F)$ over $M'$, one can construct a Hermitian vector bundle over the closed double
\[
\double M':=M'\cup_{\Sigma}M'^-,
\]
where $M'^-$ is $M'$ with opposite orientation. Precisely, using the bundle isomorphism $\frakt$, one defines a bundle $V(E,F)$ over $\double M'$ such that it coincides with $E$ over $M'$ and with $F$ over $M'^-$ outside a small collar neighborhood of $\Sigma$. Let $D_{V(E,F)}$ be the Dirac operator over $\double M'$ twisted by $V(E,F)$.

Let $\rho=\rho^+\oplus\rho^-\in C^\infty(M,U(l)\oplus U(l))$. Suppose $\rho^+=\rho^-$ is locally constant outside a Callias support of $\calD_\Phi$. Then $[D,\rho]$ is automatically a bounded bundle endomorphism. In this case one can choose the above hypersurface $\Sigma$ such that $\rho^+=\rho^-$ is locally constant near $\Sigma$ and on $M''$.
Thus $\rho^+_{|M'}$ and $\rho^-_{|M'}$ can be glued smoothly to yield $\tilde{\rho}\in C^\infty(\double M',U(l))$, which means that $\tilde{\rho}$ is an extension of $\rho^+_{|M'}$ to $\double M'$ such that $\tilde{\rho}_{|M'^-}$ coincides with $\rho^-_{|M'}$.

\begin{theorem}\label{T:codim 0 spf}
Under the above setting, if $\Phi\in\End(S)$ is a Callias potential that has a definite sign at infinity, namely $\Phi$ is a positive (or negative) operator outside a Callias support $M'$, then for any $\rho=\rho^+\oplus\rho^-\in C^\infty(M,U(l)\oplus U(l)$ such that $\rho^+=\rho^-$ is locally constant outside an interior subset of $M'$, we have
\[
\spf(\calD_\Phi,\rho)=\spf(D_{V(E,F)},\tilde{\rho}),
\]
where $D_{V(E,F)}$ is the twisted Dirac operator on $\double M'$.
\end{theorem}

\begin{proof}
Without loss of generality, suppose $\Phi$ is positive at infinity. Cut $M$ along $\Sigma$ as in \eqref{E:partition} and impose the local boundary condition associated to the admissible boundary involution $\bfs=\id\in\End(S_{|\pM^\cut})$. By Theorem~\ref{T:splitting spf},
\[
\spf(\calD_\Phi,\rho)=\spf(\calD'_{\Phi,\id},\rho')+\spf(\calD''_{\Phi,\id},\rho''),
\]
where the terms with superscript $'$ or $''$ denote the corresponding items on $M'$ or $M''$.

Since $\calD''_\Phi$ is a Callias operator with empty Callias support and $\bfs\Phi$ is non-negative on $S_{|\pM''}$, it follows from Proposition~\ref{P:spf vanishing}.\ref{IT:spf vanishing-1} that $\spf(\calD''_{\Phi,\id},\rho'')=0$. For $\spf(\calD'_{\Phi,\id},\rho')$, one can apply Lemma~\ref{L:spf homotop inv} to equate it with $\spf(\calD'_{0,\id},\rho')$. So
\begin{equation}\label{E:GL pair spf-1}
\spf(\calD_\Phi,\rho)=\spf(\calD'_{0,\id},\rho')
\end{equation}

On the other hand, denote $\tilE=V(E,F)$ and let $\tilF$ be the extension of $F$ to $\double M'$ such that $\tilF_{|M'^-}=F$. Replacing $E$ and $F$ by $\tilE$ and $\tilF$ respectively in \eqref{E:twist bundle odd} and \eqref{E:twist Dirac odd}, one gets a twisted bundle $\tilS$ and a Dirac operator $\tilD$ twisted by $\tilE\oplus\tilF^{\rm op}$ over $\double M'$. By the decomposition of $\tilS$ and $\tilD$, we have
\[
\spf(\tilD,\tilde{\rho}\oplus\tilde{\rho}^-)=\spf(D_{V(E,F)},\tilde{\rho})-\spf(D_{\tilF},\tilde{\rho}^-),
\]
where $\tilde{\rho}^-\in C^\infty(\double M',U(l))$ is the smooth gluing of $\rho^-_{|M'}$ and $\rho^-_{|M'}$.
By \cite[Corollary~2.7]{Getzler93}, the last term can be computed as
\[
\spf(D_{\tilF},\tilde{\rho}^-)=\sqrt{\frac{\varepsilon}{\pi}}\int_0^1\Tr\big((\tilde{\rho}^-)^{-1}[D_{\tilF},\tilde{\rho}^-]e^{-\varepsilon D^2_{\tilF}(r)}\big)\D r,
\]
where $D_{\tilF}(r)=(1-r)D_{\tilF}+r(\tilde{\rho}^-)^{-1}D_{\tilF}\tilde{\rho}^-$. Note that the trace can be written as an integration over $\double M'$. It follows from the symmetry of $\tilF$ and $\tilde{\rho}^-$ that $\spf(D_{\tilF},\tilde{\rho}^-)$ vanishes. Hence
\begin{equation}\label{E:GL pair spf-2}
\spf(\tilD,\tilde{\rho}\oplus\tilde{\rho}^-)=\spf(D_{V(E,F)},\tilde{\rho}).
\end{equation}

Lastly, we compute $\spf(\tilD,\tilde{\rho}\oplus\tilde{\rho}^-)$ by applying Theorem~\ref{T:splitting spf} again,
\begin{equation}\label{E:GL pair spf-3}
\spf(\tilD,\tilde{\rho}\oplus\tilde{\rho}^-)=\spf(\calD'_{0,\id},\rho')+\spf(\calD'^{F,F}_{0,\id},\rho^-_{|M'}\oplus\rho^-_{|M'})=\spf(\calD'_{0,\id},\rho'),
\end{equation}
where $\calD'^{F,F}_{0,\id}$ is the Dirac operator on $M'^-$ twisted by $F\oplus F^{\rm op}$, whose spectral flow thus vanishes because of Proposition~\ref{P:spf vanishing}.\ref{IT:spf vanishing-2}. The theorem is proved by combining \eqref{E:GL pair spf-1}, \eqref{E:GL pair spf-2} and \eqref{E:GL pair spf-3}.
\end{proof}

The notion of Gromov--Lawson pair also makes sense on a manifold with compact boundary, with the compact set $K$ in Assumption~\ref{A:GL pair} contained in the interior of $M$. In this setting, the proof of Theorem~\ref{T:codim 0 spf} indicates Corollary~\ref{IntroC:codim 0 spf}, a codimension zero spectral flow formula on compact manifolds with boundary.

\subsection{A codimension one spectral flow formula on even-dimensional manifolds}\label{SS:codim 1}

In this subsection, we turn our attention to the even-dimensional case. On an even-dimensional complete Riemannian manifold $M$ without boundary, let $\uS$ be a Dirac bundle\footnote{In applications, $\uS$ is usually $\ZZ_2$-graded.} with Dirac operator $\uD$. Form a $\ZZ_2$-graded Dirac bundle
\begin{equation}\label{E:graded bundle-even}
S:=\uS\oplus\uS,
\end{equation}
with respect to which the Clifford multiplication is given by
\[
\rmc(\xi)=\left(
\begin{matrix}
0 & \underline{{\rmc}}(\xi) \\
\underline{{\rmc}}(\xi) & 0
\end{matrix}\right),\quad\xi\in T^*M,
\]
where $\underline{\rmc}(\xi)$ is the Clifford multiplication on $\uS$.
Then $S$ can be made a relative Dirac bundle with \emph{empty} support by setting
\[
\theta=\left(
\begin{matrix}
1 & 0 \\
0 & -1
\end{matrix}\right)\in\End(S).
\]
The Dirac operator associated to $S$ has the form
\begin{equation}\label{E:Dirac even}
D=\left(
\begin{matrix}
0 & \uD \\
\uD & 0
\end{matrix}\right).
\end{equation}

Let $\phi$ be a bundle endomorphism on $\uS$ which commutes with $\underline{\rmc}(\cdot)$. Like in Definition~\ref{D:Callias}, suppose the bundle endomorphism $\phi^2-|[\uD,\phi]|$ is uniformly positive outside a compact subset $K\Subset M$. Then we can construct a Callias potential $\Phi:=\phi\oplus\phi\in\End(S)$ with respect to \eqref{E:graded bundle-even}, and get a Callias operator
\begin{equation}\label{E:Callias even}
\calD_\Phi=D+\Phi\theta=\left(
\begin{matrix}
\phi & \uD \\
\uD & -\phi
\end{matrix}\right).
\end{equation}
It has $K$ as a Callias support.

Let $M'\Subset M$ be a compact subset containing $K$ in the interior such that $\Sigma:=\pM'$ is a closed hypersurface of $M$ with trivial normal bundle. Since $\phi$ is invertible on $\Sigma$, there is a bundle decomposition $\uS_\Sigma=\uS_{\Sigma+}\oplus\uS_{\Sigma-}$, where $\uS_\Sigma:=\uS_{|\Sigma}$ and $\uS_{\Sigma\pm}$ denote the positive/negative eigenspaces of $\phi_{|\Sigma}\in\End(\uS_{\Sigma})$. $\uS_\Sigma$ and $\uS_{\Sigma\pm}$ are Dirac bundles over $\Sigma$ with the Clifford multiplication $\rmc_\Sigma(\xi):=\underline{\rmc}(\nu^*)^{-1}\underline{\rmc}(\xi)$ for $\xi\in T^*\Sigma$, where $\nu$ is the unit normal vector field pointing inward $M'$. Let $D_\Sigma:C^\infty(\Sigma,\uS_\Sigma)\to C^\infty(\Sigma,\uS_\Sigma)$ and $D_{\Sigma\pm}:C^\infty(\Sigma,\uS_{\Sigma\pm})\to C^\infty(\Sigma,\uS_{\Sigma\pm})$ be the corresponding Dirac operators. Then
\[
D_{\Sigma\pm}=P_\pm\circ D_\Sigma\circ P_\pm,
\]
where $P_\pm$ is the projection from $\uS_\Sigma$ onto $\uS_{\Sigma\pm}$.

\begin{theorem}\label{T:codim 1 spf}
Under the above setting, for any $\rho\in C^\infty(M,U(l))$ satisfying Assumption~\ref{A:rho rel bounded}, we have
\begin{equation}\label{E:codim 1 spf}
\spf(\calD_\Phi,\rho)=\frac{1}{2}\left(\spf(D_{\Sigma+},\rho_{|\Sigma})-\spf(D_{\Sigma-},\rho_{|\Sigma})\right)=\spf(D_{\Sigma+},\rho_{|\Sigma})=-\spf(D_{\Sigma-},\rho_{|\Sigma}).
\end{equation}
\end{theorem}

\begin{proof}
Cutting $M$ along $\Sigma$, one gets a manifold with boundary $M^\cut=M'\sqcup_\Sigma M''$. Denote $\Sigma'$ (resp. $\Sigma''$) to be the boundary of $M'$ (resp. $M''$). Choose an admissible boundary involution $\bfs\in\End(S_{|\pM^\cut})$ such that $\bfs_{|\Sigma'}=\bfs_{|\Sigma''}=\phi_0\oplus\phi_0$
with respect to \eqref{E:graded bundle-even}. Here $\phi_0=\phi_{|\Sigma}(\phi_{|\Sigma}^2)^{-1/2}$ is the unitarization of $\phi_{|\Sigma}$, i.e., $\phi_0=\pm\id$ on $\uS_{\Sigma\pm}$.

By Theorem~\ref{T:splitting spf},
\[
\spf(\calD_\Phi,\rho)=\spf(\calD'_{\Phi,\bfs},\rho')+\spf(\calD''_{\Phi,\bfs},\rho'').
\]
Note that $\calD_\Phi$ has empty Callias support on $M''$ and $\bfs\Phi$ is non-negative on $S_{\Sigma''}$. Hence $\spf(\calD''_{\Phi,\bfs},\rho'')=0$ from Proposition~\ref{P:spf vanishing}.\ref{IT:spf vanishing-1}.

For $\spf(\calD'_{\Phi,\bfs},\rho')$, since it is independent of the Callias potential by Lemma~\ref{L:spf homotop inv}, we may assume that $\Phi=0$. Now let us take a closer look at the boundary condition. Note that the boundary condition is determined by
\[
\chi=\bfs\rmc(\nu^*)\theta=\left(
\begin{matrix}
0 & -\underline{\rmc}(\nu^*)\phi_0 \\
\underline{\rmc}(\nu^*)\phi_0 & 0
\end{matrix}\right).
\]
If $u\oplus v\in\dom(\calD'_{0,\bfs})$, then $\chi(u_{|\Sigma'}\oplus v_{|\Sigma'})=u_{|\Sigma'}\oplus v_{|\Sigma'}$, which implies that $v_{|\Sigma'}=\underline{\rmc}(\nu^*)\phi_0u_{|\Sigma'}$. This is exactly the local boundary condition studied by Gorokhovsky and Lesch in \cite{GorokhovskyLesch15}. Using \cite[Theorems~3.3 and 3.3$'$]{GorokhovskyLesch15}, we obtain \eqref{E:codim 1 spf}, where the last two equalities are due to the cobordism invariance of spectral flow (\cite[Corollary~3.5]{GorokhovskyLesch15}).
\end{proof}

An interesting application of the codimension one spectral flow formula is Corollary~\ref{IntroC:codim 1 spf} for the Riemannian band, which we restate here.

\begin{corollary}\label{C:codim 1 spf}
Let $M$ be an even-dimensional compact Riemannian band, and $D$ be the Dirac operator \eqref{E:Dirac even} on $S=\uS\oplus\uS$. Choose the admissible boundary involution $\bfs=\pm\id\in\End(S_{|\pM})$ on $\p_\pm M$. Then for any Callias potential $\Phi\in\End(S)$ and any $\rho\in C^\infty(M,U(l))$,
\[
\spf(\calD_{\Phi,\bfs},\rho)=\spf(D_{\p_+M},\rho_{|\p_+M})=-\spf(D_{\p_-M},\rho_{|\p_-M}).
\]
\end{corollary}

\begin{proof}
Since the spectral flow does not depend on the Callias potential, we can choose a Callias potential $\Phi=\phi\oplus\phi$ such that $\phi_{|\p_\pm M}=\pm\id$. Then the admissible boundary involution $\bfs$ is given exactly by the unitarization of $\phi_{|\p_\pm M}$. The formula for the spectral flow now follows from the second half of the proof of Theorem~\ref{T:codim 1 spf}.
\end{proof}

\begin{remark}\label{R:codim 1 spf}
The main scenario we shall consider is that $M$ is an even-dimensional Riemannian \emph{spin} manifold. In this case the Dirac bundle $\uS\to M$ is a $\ZZ_2$-graded twisted spinor bundle $\slaS\otimes E=(\slaS^+\otimes E)\oplus(\slaS^-\otimes E)$ (where $E$ is a Hermitian vector bundle with a Hermitian connection), and $\uD$ is the $\ZZ_2$-graded twisted spin Dirac operator $\slaD_E=\slaD^+_E\oplus\slaD^-_E$. For an oriented hypersurface $\Sigma\subset M$ as in Theorem~\ref{T:codim 1 spf} (or for $\pM$ in the setting of Corollary~\ref{C:codim 1 spf}), let $\slaS_\Sigma$ denote the complex spinor bundle over $\Sigma$. Then $\slaS_\Sigma$ is isomorphic to both $\slaS^+_{|\Sigma}$ and $\slaS^-_{|\Sigma}$. In particular, the restriction of $\slaD_E$ to $\Sigma$ is isomorphic to two copies of $\slaD_{\Sigma,E_{|\Sigma}}$. 

In particular, for the case of Corollary~\ref{C:codim 1 spf}, our formula becomes
\[
\spf(\calD_{\Phi,\bfs},\rho)=2\spf(\slaD_{E_{|\p_+M}},\rho_{|\p_+M})=-2\spf(\slaD_{E_{|\p_-M}},\rho_{|\p_-M}).
\]
\end{remark}

\begin{remark}\label{R:Lip potential}
As mentioned earlier, we can include Lipschitz continuous Callias potentials to the discussion of the spectral flow formulas. To be precise, we allow the Callias potential $\Phi$ to be a Lipschitz continuous bundle endomorphism on a compact subset of $M$. In this case, one can always find a smooth Callias potential $\Phi'$ which coincides with $\Phi$ outside a compact subset. By Lemma~\ref{L:spf homotop inv}, the spectral flow remains fixed with $\Phi$ replaced by $\Phi'$. In other words, all the results in this section remain true for such potentials. The motivation of considering such kinds of Callias potentials is the distance-related estimates proposed in \cite{CeccZeid24GT}.
\end{remark}

\section{K-cowaist on odd-dimensional manifolds}\label{S:K-cw odd}

In this section, we introduce an intrinsic definition of K-cowaist on odd-dimensional manifolds. What we shall use is the odd Chern character of a smooth map from the manifold to a unitary group. One will see that our odd K-cowaist is a natural extension of the original (even) K-cowaist.

\subsection{K-cowaist on even-dimensional manifolds}\label{SS:K-cw even}

We review the notion of K-cowaist on even-dimensional manifolds. For details, we refer the readers to \cite{Gromov96,Min-Oo02,Gromov17-101,BaerHanke23,CeccZeid24GT,Gromov23Four,Wang23K-cw,BaerHanke24K-cw}, etc.

Let $M$ be an oriented closed Riemannian manifold of even dimension and $E\to M$ be a Hermitian vector bundle with a Hermitian connection $\nabla^E$. We say that $E$ has a non-trivial Chern number if there exist positive integers $\gamma_1,\cdots,\gamma_m$, such that
\[
\int_M\ch_{\gamma_1}(E)\wedge\cdots\wedge\ch_{\gamma_m}(E)\ne0,
\]
where $\ch_{\gamma_j}(E)$ denotes the $\gamma_j$-th Chern character form of $E$.
Equivalently, one can replace $\ch_{\gamma_j}(E)$ by the Chern class form $c_{\gamma_j}(E)$ of $E$. Let $R^E=(\nabla^E)^2$ be the curvature of the connection $\nabla^E$, with the norm given by
\[
\|R^E\|:=\sup_{p\in M}\sup_{|v\wedge w|=1}|R^E(v\wedge w)|,
\]
where $v,w\in T_pM$ and $|R^E(v\wedge w)|$ is the operator norm of $R^E(v\wedge w)\in\End(E_p)$.

\begin{definition}[{\cite[Section~4]{Gromov96}}]\label{D:K-cw even}
The \emph{K-cowaist} of an even-dimensional oriented closed Riemannian manifold $M$ is defined to be
\[
\Kcw(M):=\left(\inf\left\{\|R^E\|\;|\;E\mbox{ has a non-trivial Chern number}\right\}\right)^{-1}\in(0,\infty].
\]
\end{definition}

The notion of K-cowaist can also be set up on non-compact manifolds or manifolds with boundary. In these situations we call $E$ a \emph{compatible} bundle if $E$ is isomorphic to the trivial bundle with trivial connection at infinity and near the boundary.

\begin{definition}\label{D:K-cw compatible even}
Let $M$ be an oriented (not necessarily compact) Riemannian manifold of even dimension possibly with compact boundary. The \emph{K-cowaist} of $M$ is defined to be
\[
\Kcw(M):=(\inf\|R^E\|)^{-1}\in[0,\infty]
\]
with $E$ ranging over all compatible bundles having a non-trivial Chern number.
\end{definition}

\begin{remark}\label{R:K-cw compatible even}
Another notion, which is called K-cowaist$^+$ by Gromov in \cite{Gromov96,Gromov17-101}, is defined to be $\Kcw^+(M):=(\inf\|R^{E\oplus F}\|)^{-1}$ with $(E,F)$ ranging over all compatible pairs of bundles having a non-trivial relative Chern number. Here $(E,F)$ is called a \emph{compatible pair} if $E$ and $F$ are isomorphic (as two Hermitian vector bundles with Hermitian connections) at infinity and near the boundary. We say a compatible pair of bundles $(E,F)$ having a non-trivial relative Chern number if there exist positive integers $\gamma_1,\cdots,\gamma_m$, such that
\[
\int_M\left(\ch_{\gamma_1}(E)\wedge\cdots\wedge\ch_{\gamma_m}(E)-\ch_{\gamma_1}(F)\wedge\cdots\wedge \ch_{\gamma_m}(F)\right)\ne0.
\]
When $F$ is the trivially flat bundle, this reduces to the condition of Definition~\ref{D:K-cw compatible even}.
It is clear that $\Kcw^+(M)\ge\Kcw(M)$ and they are the same on closed manifolds. Gromov pointed out that $\Kcw^+(M)$ can be significantly greater than $\Kcw(M)$ for open manifolds $M$.
\end{remark}

\begin{remark}\label{R:K-cw metric}
The K-cowaist depends on the Riemannian metric. But on compact manifolds, the property of having infinite K-cowaist is independent of the metric.
\end{remark}

Using similar ideas, one can also talk about $\omega$-cowaist.

\begin{definition}[\cite{BaerHanke24K-cw}]\label{D:omega-cw even}
Let $M$ be an oriented (not necessarily compact) Riemannian manifold of even dimension $2m$ possibly with compact boundary, and $\omega=1+\omega_1+\cdots+\omega_m$ be a smooth differential form on $M$, where $\omega_j$ has degree $2j$. The \emph{$\omega$-cowaist} of $M$ is defined by
\[
\ocw(M):=\left(\inf\|R^E\|\right)^{-1}\in[0,\infty]
\]
with $E$ ranging over all compatible bundles such that
\[
\int_M\omega\wedge\left(\ch(E)-\rk(E)\right)\ne0.
\]

An important special case is $\omega$ being the $\hat{A}$-genus form of $M$, and we call it \emph{$\hat{A}$-cowaist}, denoted by $\hAcw(M)$.
\end{definition}

The K-cowaist and $\omega$-cowaist satisfy relationship stated in Theorem~\ref{IntroT:K-cw omega-cw}. And their role in studying PSC problems on spin manifolds is stated in Theorem~\ref{IntroT:K-cw psc even}.

\subsection{Chern--Simons forms and odd Chern characters}\label{SS:Chern-Simons}

In this subsection, we first review the notion of odd Chern characters, mainly according to \cite{Getzler93,Zhang01book}, and then extend the Adams operation to this case. This will be used in the discussion of odd K-cowaist later.

Let $M$ be an oriented closed manifold and $E\to M$ be a complex vector bundle with connection $\nabla$. Recall the Chern character form associated to $\nabla$ is the even-degree closed differential form defined by
\[
\ch(E,\nabla):=\tr\left(\exp\Big(\frac{\sqrt{-1}}{2\pi}\nabla^2\Big)\right).
\]
For two connections $\nabla_0,\nabla_1$ on $E$, their \emph{Chern--Simons transgressed form} associated to the Chern character is the odd-degree differential form
\[
\Tch(E,\nabla_0,\nabla_1)=-\int_0^1\tr\left(\frac{\sqrt{-1}}{2\pi}\dot{\nabla}_r\exp\Big(\frac{\sqrt{-1}}{2\pi}\nabla_r^2\Big)\right)\D r,
\]
where $\nabla_r=(1-r)\nabla_0+r\nabla_1$ and $\dot{\nabla}_r=\nabla_1-\nabla_0$. It satisfies the transgression formula (cf. \cite[Chapter~1]{Zhang01book})
\begin{equation}\label{E:transgression}
\ch(E,\nabla_0)-\ch(E,\nabla_1)=\D\Tch(E,\nabla_0,\nabla_1).
\end{equation}
If $\nabla_0$ and $\nabla_1$ are both flat connections, then $\Tch(E,\nabla_0,\nabla_1)$ is a closed form. Let $\Omega=\nabla_1-\nabla_0$. Then it can be derived like \cite[Section~1]{Getzler93} that
\begin{equation}\label{E:Chern-Simons}
\Tch(E,\nabla_0,\nabla_1)=\sum_{j=1}^\infty\Big(\frac{1}{2\pi\sqrt{-1}}\Big)^{j}\frac{(j-1)!}{(2j-1)!}\tr(\Omega^{2j-1}).
\end{equation}

Consider a special case that $E$ is the trivial bundle $M\times\CC^l$ with a trivial connection $\D$. Then any $\rho\in C^\infty(M,U(l))$ induces a family of connections
\begin{equation}\label{E:conn family}
\nabla_r:=\D+r\rho^{-1}(\D\rho),\quad r\in[0,1]
\end{equation}
on $E$. In this setting, the Chern--Simons form \eqref{E:Chern-Simons} is just the \emph{odd Chern character} form (cf. \cite[Chapter~1]{Zhang01book})
\[\label{E:odd chern ch}
\ch(\rho,\D)=\sum_{j=1}^\infty\ch_j(\rho,\D):=\sum_{j=1}^\infty\Big(\frac{1}{2\pi\sqrt{-1}}\Big)^{j}\frac{(j-1)!}{(2j-1)!}\tr\big((\rho^{-1}(\D\rho))^{2j-1}\big),
\]
where the $j$-th odd Chern character form $\ch_j(\rho,\D)$ is closed of degree $2j-1$. (We adopt the convention that $\ch_0(\rho,\D)=0$.) A straightforward computation yields that the curvature $R^{\nabla_r}$ of the connection $\nabla_r$ is $-r(1-r)(\rho^{-1}(\D\rho))^2$. We set the curvature associated to $\rho$ to be
\begin{equation}\label{E:curv rho}
R^\rho:=\frac{1}{4}(\rho^{-1}(\D\rho))^2.
\end{equation}
So that $\|R^{\nabla_r}\|\le\|R^\rho\|$ for any $r\in[0,1]$.

In the proof of Theorem~\ref{IntroT:K-cw omega-cw} for even-dimensional manifolds, an operation in the even K-group $K(M)$ called Adams operation is involved. Briefly speaking (cf. \cite{Adams62,Atiyah89K-book,Karoubi78K-book}), for a Hermitian vector bundle $E\to M$ with a Hermitian connection and an integer $k\ge0$, the $k$-th Adams operation $\psi_k(E)$ is a virtual bundle $\psi_k^+(E)-\psi_k^-(E)$ constructed from certain universal combinations of direct sums, tensor products, and exterior products of $E$. It has the property that
\begin{equation}\label{E:Adams op even}
\ch_j(\psi_k(E))=\ch_j(\psi_k^+(E))-\ch_j(\psi_k^-(E))=k^j\ch_j(E),
\end{equation}
and $\|R^{\psi_k^\pm(E)}\|\le c_k\|R^E\|$ for some constant $c_k$ depending only on $k$.

For the aforementioned case that $E=M\times\CC^l$ with a trivial connection $\D$, and $\rho\in C^\infty(M,U(l))$. Performing the $k$-th Adams operation to $E$ with the connection $\nabla_1=\D+\rho^{-1}(\D\rho)$, we get two bundles $\psi_k^\pm(E)$ with connections denoted by $\psi_k^\pm(\nabla_1)$. Since taking inverse or differential commutes with the operation of taking direct sum, tensor product or exterior product, one deduces that $\psi_k^\pm(\nabla_1)$ can be expressed in the form of
\begin{equation}\label{E:Adams op conn}
\psi_k^\pm(\nabla_1)=\tilde{\D}^\pm+(\tilde{\rho}^\pm)^{-1}(\tilde{\D}^\pm\tilde{\rho}^\pm),
\end{equation}
where $\tilde{\D}^\pm=\psi_k^\pm(\D)$ are the trivial connections on $\psi_k^\pm(E)$ corresponding to $\D$, and $\tilde{\rho}^\pm$ are smooth functions on $M$ with values in some unitary group.

\begin{definition}\label{D:Adams op odd}
We define the $k$-th Adams operation to $\rho\in C^\infty(M,U(l))$, denoted by $\psi_k(\rho)$, to be the virtual map $\psi_k^+(\rho)-\psi_k^-(\rho)$, where $\psi_k^\pm(\rho)$ are given by the maps $\tilde{\rho}^\pm$ in \eqref{E:Adams op conn}.
\end{definition}

\begin{remark}\label{R:Adams op odd Chern}
From \eqref{E:transgression} and \eqref{E:Adams op even}, we obtain that
\begin{equation}\label{E:Adams op odd}
\ch_j(\psi_k(\rho),\psi_k(\D))=\ch_j(\psi_k^+(\rho),\psi_k^+(\D))-\ch_j(\psi_k^-(\rho),\psi_k^-(\D))=k^j\ch_j(\rho,\D).
\end{equation}
It also holds that $\|R^{\psi_k^\pm(\rho)}\|\le c_k\|R^\rho\|$ for some constant $c_k$ depending only on $k$.
\end{remark}

\begin{remark}\label{R:Adams op odd-2}
Since an element of $K^{-1}(M)$ may be represented by $\rho\in C^\infty(M,U(l))$, Definition~\ref{D:Adams op odd} can be seen as a generalization of the Adams operation to the odd K-group $K^{-1}(M)$.
\end{remark}

\begin{remark}\label{R:odd Chern non-closed}
If the manifold $M$ is non-compact and/or with boundary, we consider a map $\rho\in C^\infty(M,U(l))$ which is locally constant at infinity and near the boundary. The above discussion is still valid.
\end{remark}

\subsection{K-cowaist on odd-dimensional manifolds}\label{SS:K-cw odd}

In this subsection, we extend the notion of K-cowaist and $\omega$-cowaist to odd-dimensional manifolds in an intrinsic way.

It is clear that any vector bundle over an odd-dimensional manifold $M$ has trivial Chern numbers. Therefore Definitions~\ref{D:K-cw even} and \ref{D:K-cw compatible even} do not make sense in odd-dimensional situation. Previously, this was settled by converting to even-dimensional case via taking product with another space. In literature, there are different conventions, so that the K-cowaist of $M$ might be defined as one of the following:
\begin{itemize}
\item $\Kcw(M\times\RR)$;
\item $\Kcw(M\times\rmS^1)$;
\item $\sup_{k}\Kcw(M\times\RR^k)$, where $k\ge1$ is odd;
\item $\sup_{k}\Kcw(M\times\rmT^k)$, where $k\ge1$ is odd.
\end{itemize}
It appears that these definitions are not equivalent (cf. \cite[Section~12]{Gromov17-101}). Here we give a more intrinsic definition using odd Chern characters.

A smooth map $\rho:M\to U(l)$ is called \emph{compatible} if it is trivial, i.e., locally constant at infinity and near the boundary. A compatible bundle $E\to M$ and a compatible map $\rho\in C^\infty(M,U(l))$ are called an \emph{admissible pair} if they have a non-trivial odd Chern number, i.e., if there exist non-negative integers $\gamma_1,\cdots,\gamma_m$ such that
\begin{equation}\label{E:admi pair}
\int_M\ch_{\gamma_1}(E)\wedge\cdots\wedge\ch_{\gamma_{m-1}}(E)\wedge\ch_{\gamma_m}(\rho)\ne0.
\end{equation}
And they are called an \emph{$\omega$-admissible pair} if
\begin{equation}\label{E:omega admi pair}
\int_M\omega\wedge\left(\ch(E)-\rk(E)\right)\wedge\ch(\rho)\ne0,
\end{equation}
where $\omega$ is a mixed differential form of even degree. Here we omit the trivial connection $\D$ in the odd Chern character form as its cohomology class is independent of $\D$ (cf. \cite[Section~1.8]{Zhang01book}).

\begin{definition}\label{D:K-cw odd}
Let $M$ be an oriented (not necessarily compact) Riemannian manifold of odd dimension $2m-1$ possibly with compact boundary. The \emph{K-cowaist} of $M$ is defined to be
\[
\Kcw(M):=\left(\inf\{\|R^E\|+\|R^\rho\|\}\right)^{-1}\in[0,\infty]
\]
with $(E;\rho)$ ranging over all admissible pairs. (Recall the curvature $R^\rho$ is defined in \eqref{E:curv rho}.)

Similarly, let $\omega=1+\omega_1+\cdots+\omega_{m-1}$ be a smooth differential form on $M$, where $\omega_j$ has degree $2j$. The \emph{$\omega$-cowaist} of $M$ is defined by
\[
\ocw(M):=\left(\inf\{\|R^E\|+\|R^\rho\|\}\right)^{-1}\in[0,\infty]
\]
with $(E;\rho)$ ranging over all $\omega$-admissible pairs.
\end{definition}

\begin{remark}\label{R:admi pair}
Note that the appearance of terms involving $\rho$ makes the left-hand sides of \eqref{E:admi pair} and \eqref{E:omega admi pair} vanish on even-dimensional manifolds. Therefore $\rho$ is not needed in Definitions~\ref{D:K-cw compatible even} and \ref{D:omega-cw even}.
\end{remark}

\begin{remark}\label{R:K-cw odd}
While our definition of odd K-cowaist is directly on $M$, there is a way to relate it to $M\times\rmS^1$. Basically, given an admissible pair $(E;\rho)$ over $M$, one has the pullback bundle $E\otimes\CC^l$ over $M\times[0,1]$. Multiplication by $\rho$ induces an identification of the restrictions of $E\otimes\CC^l$ to the endpoints of the interval, so that we get a vector bundle $\tilE\to M\times\rmS^1$. Consider such bundles $\tilE$ with connections (and only such connections) of the form
\[
\nabla^{\tilE}_{|M\times\{r\}}=\nabla^E\otimes\id+\id\otimes(\D+r\rho^{-1}[\D,\rho])
\]
as in Remark~\ref{R:family conn}. Then $R^{\tilE}=R^E-4r(1-r)R^\rho$. Also one can check that \eqref{E:admi pair} implies that $\tilE$ has a non-trivial Chern number. Hence $\Kcw(M)\le\Kcw(M\times\rmS^1)$.
\end{remark}

Theorem~\ref{IntroT:K-cw omega-cw} can be established under the above definition.

\begin{theorem}\label{T:K-cw omega-cw odd}
Let $M$ be an oriented Riemannian manifold of odd dimension $2m-1$ possibly with compact boundary. Let $\omega=1+\omega_1+\cdots+\omega_{m-1}$ be a smooth differential form on $M$, where $\omega_j$ has degree $2j$.
There exists a constant $c(m)$ depending only on $m$ such that $\Kcw(M)\le c(m)\cdot\ocw(M)$.

In particular, if $M$ has infinite K-cowaist, then it has infinite $\omega$-cowaist for any $\omega$.
\end{theorem}

The idea of the proof is from \cite{BaerHanke24K-cw}. For completeness, we include it below.

\begin{proof}
If there are no admissible pairs over $M$, then $\Kcw(M)=0$. The theorem holds automatically.

Otherwise, we first show a claim.

\textbf{Claim.} Let $(E;\rho)$ be an admissible pair over $M$. Then there always exists an $\omega$-admissible pair $(E';\rho')$ constructed from $(E;\rho)$.

For a multi-index $k=(k_1,\cdots,k_{m-1})$ of non-negative integers and a non-negative integer $k_m$,  rewrite
\[
\psi_{k}(E):=\psi_{k_1}(E)\otimes\cdots\otimes\psi_{k_{m-1}}(E)\quad\mbox{and}\quad\psi_{k_m}(\rho)
\]
as differences of honest bundles/maps
\[
\begin{aligned}
\psi_{k}(E)& =\bigoplus_{\#\{-\}={\rm even}}\psi_{k_1}^\pm(E)\otimes\cdots\otimes\psi_{k_{m-1}}^\pm(E)-\bigoplus_{\#\{-\}={\rm odd}}\psi_{k_1}^\pm(E)\otimes\cdots\otimes\psi_{k_{m-1}}^\pm(E) \\
& =:\psi_{k}^+(E)-\psi_{k}^-(E), \\
\psi_{k_m}(\rho)& =\psi_{k_m}^+(\rho)-\psi_{k_m}^-(\rho).
\end{aligned}
\]

Put
\[
P(k,k_m):=\int_M\omega\wedge(\ch(\psi_{k}(E))-\rk(\psi_{k}(E)))\wedge\ch(\psi_{k_m}(\rho)).
\]
Extending $\omega=1+\omega_1+\cdots+\omega_{m-1}$ and using \eqref{E:Adams op even}, \eqref{E:Adams op odd}, we get
\[
P(k,k_m)=\sum_{\gamma_1+\cdots+\gamma_m=m}k_1^{\gamma_1}\cdots k_m^{\gamma_m}\ch_{\gamma_1}(E)\wedge\cdots\wedge\ch_{\gamma_{m-1}}(E)\wedge\ch_{\gamma_m}(\rho)+{\rm l.o.t.}
\]
where ``l.o.t.'' stands for terms of lower total order in $k_1,\cdots,k_m$. So $P$ is a polynomial in $k_1,\cdots,k_m$ of total degree at most $m$.

If $P(k,k_m)$ vanishes for all $(k_1,\cdots,k_m)\in\{0,1,\cdots,m\}^m$, then $P$ would vanish as a polynomial, so
\[
\int_M\ch_{\gamma_1}(E)\wedge\cdots\wedge\ch_{\gamma_{m-1}}(E)\wedge\ch_{\gamma_m}(\rho)=0
\]
for all non-negative integers $\gamma_j$ with $\gamma_1+\cdots+\gamma_m=m$, contradicting the fact that $(E;\rho)$ is an admissible pair. Hence we can choose $(k_1,\cdots,k_m)\in\{0,1,\cdots,m\}^m$ such that $P(k,k_m)\ne0$, which means
\[
\begin{aligned}
0\ne& \int_M\omega\wedge(\ch(\psi_{k}(E))-\rk(\psi_{k}(E)))\wedge\ch(\psi_{k_m}(\rho)) \\
=& \int_M\omega\wedge(\ch(\psi_{k}^+(E))-\rk(\psi_{k}^+(E)))\wedge\ch(\psi_{k_m}^+(\rho)) \\
& +\int_M\omega\wedge(\ch(\psi_{k}^-(E))-\rk(\psi_{k}^-(E)))\wedge\ch(\psi_{k_m}^-(\rho)) \\
& -\int_M\omega\wedge(\ch(\psi_{k}^+(E))-\rk(\psi_{k}^+(E)))\wedge\ch(\psi_{k_m}^-(\rho)) \\
& -\int_M\omega\wedge(\ch(\psi_{k}^-(E))-\rk(\psi_{k}^-(E)))\wedge\ch(\psi_{k_m}^+(\rho)).
\end{aligned}
\]
In particular, we can find a compatible bundle $E'$ and a compatible map $\rho'$ which satisfy \eqref{E:omega admi pair}. This proves the Claim.

Since $(k_1,\cdots,k_m)$ can only achieve finitely many values, we have
\[
\|R^{E'}\|+\|R^{\rho'}\|\le c'(m)\left(\|R^{E}\|+\|R^{\rho}\|\right)
\]
for some constant $c'(m)$. The theorem now follows from taking infimum over all admissible pairs $(E;\rho)$.
\end{proof}

\subsection{K-cowaist of finite coverings}\label{SS:K-cw covering}

An important property of K-cowaist is that it is preserved under finite coverings.

\begin{proposition}\label{P:K-cw covering}
Let $M$ be an odd-dimensional oriented complete Riemannian manifold possibly with compact boundary and $\pi:\widetilde{M}\to M$ be a finite covering of $M$ with lifted metric. Then $\Kcw(M)=\Kcw(\widetilde{M})$. The same conclusion holds for the $\omega$-cowaist.
\end{proposition}

\begin{proof}
We first show that $\Kcw(M)\ge\Kcw(\widetilde{M})$. To this end, we can assume $\Kcw(\widetilde{M})>0$. Then there exists an admissible pair $(\tilE,\tilde{\rho})$ over $\widetilde{M}$. Construct a vector bundle $E\to M$ and a map $\rho\in C^\infty(M,U(l))$ such that
\[
E_p=\bigoplus_{\tilde{p}\in\pi^{-1}(p)}\tilE_{\tilde{p}},\qquad
\rho_p=\bigoplus_{\tilde{p}\in\pi^{-1}(p)}\tilde{\rho}_{\tilde{p}}.
\]
It is easy to see that $(E;\rho)$ is also an admissible pair over $M$. Note that $\|R^E\|=\|R^{\tilE}\|$, $\|R^\rho\|=\|R^{\tilde{\rho}}\|$. Taking infimum over all admissible pairs over $M$ (resp. $\widetilde{M}$), we obtain
\[
\inf\{\|R^E\|+\|R^\rho\|\}\le\inf\{\|R^{\tilE}\|+\|R^{\tilde{\rho}}\|\}.
\]
The inequality for K-cowaists then follows from taking reciprocals.

For the other direction, we can assume that $\Kcw(M)>0$. In this case it is clear that for each admissible pair $(E;\rho)$ over $M$, the pull-back $(\pi^*(E),\pi^*(\rho))$ is an admissible pair over $\widetilde{M}$. So the opposite inequality can be derived similarly.
\end{proof}

\subsection{K-cowaist subject to a subset}\label{SS:K-cw subset}

In this subsection, we introduce a notion that will be used in next section. Let $M$ be an odd-dimensional oriented Riemannian manifold possibly with compact boundary. For a compact subset $K\Subset M^\circ$, we can consider all admissible pairs $(E;\rho)$ that are trivial outside $K$. In this case, $(E;\rho)$ is called an admissible pair with support $K$.

\begin{definition}\label{D:K-cw subset}
For a compact subset $K\Subset M^\circ$, the \emph{K-cowaist of $M$ subject to $K$} is defined to be
\[
\Kcw(M|K):=\left(\inf\{\|R^E\|+\|R^\rho\|\}\right)^{-1}\in[0,\infty]
\]
with $(E;\rho)$ ranging over all admissible pairs with support $K$. In a similar manner, we can define the $\omega$-cowaist of $M$ subject to $K$, denoted by $\ocw(M|K)$.
\end{definition}

\begin{remark}\label{R:K-cw subset}
This definition is valid for even-dimensional manifolds as well. It is clear that $\Kcw(M)\ge\Kcw(M|K)$ and they are equal when $M$ is closed and $K=M$. In general,
\[
\Kcw(M)=\sup_{K\Subset M^\circ}\Kcw(M|K).
\]
The same is true for $\omega$-cowaist.
\end{remark}

\section{Odd K-cowaist and positive scalar curvature}\label{S:odd K-cw psc}

To give an impression on the validity and advantage of our definition for the odd K-cowaist and $\omega$-cowaist, we investigate the interaction between $\hat{A}$-cowaist and scalar curvature on odd-dimensional spin manifolds in this section.

The main tool on a complete Riemannian spin manifold $(M,g)$ is the spin Dirac operator $\slaD$ on the complex spinor bundle $\slaS\to M$, or more generally, the twisted spin Dirac operator $\slaD_E$ on the twisted spinor bundle $\slaS\otimes E\to M$ for some Hermitian vector bundle $E$ with a Hermitian connection. What plays a key role is the Lichnerowicz formula
\[
\slaD_E^2=(\nabla^{\slaS\otimes E})^*\nabla^{\slaS\otimes E}+\frac{1}{4}\scal_g+\calR^E,
\]
where $\calR^E=\rmc(R^E):=\sum_{i<j}\rmc(e_i^*)\rmc(e_j^*)R^E(e_i,e_j)$ is a curvature endomorphism of the curvature $R^E$ of $E$. Moreover, if $\slaD_E(r)$ ($r\in[0,1]$) is a family of twisted spin Dirac operators induced by $\rho\in C^\infty(M,U(l))$ (cf. Remark~\ref{R:family conn}), then
\begin{equation}\label{E:Lich for}
\slaD_E^2(r)=(\nabla^*\nabla)(r)+\frac{1}{4}\scal_g+\calR^E+\calR^\rho(r),
\end{equation}
where $\calR^\rho(r)=-4r(1-r)\calR^\rho=-4r(1-r)\sum_{i<j}\rmc(e_i^*)\rmc(e_j^*)R^\rho(e_i,e_j)$, with $R^\rho$ defined by \eqref{E:curv rho}.

\subsection{Upper bound of K-cowaist on manifolds with PSC}\label{SS:up bound K-cw}

We prove the odd-dimensional analogue of Theorem~\ref{IntroT:K-cw psc even}.

\begin{theorem}\label{T:K-cw psc odd}
Let $(M,g)$ be an $n$-dimensional ($n\ge3$ odd) complete Riemannian spin manifold with (possibly empty) compact mean convex (i.e., mean non-negative) boundary. If $\scal_g\ge\sigma^2n(n-1)$ for some $\sigma>0$, then
\begin{equation}\label{E:hA-cw psc odd}
\hAcw(M)\le2\sigma^{-2},
\end{equation}
and therefore $\Kcw(M)\le C(n)\sigma^{-2}$ for some constant $C(n)$ depending only on $n$.

In particular, if $M$ has infinite K-cowaist, then it cannot admit a complete Riemannian metric of uniformly positive scalar curvature and non-negative mean curvature along the boundary.
\end{theorem}

\begin{proof}
We prove for the case that $\pM\ne\emptyset$, otherwise it is simpler. By Theorem~\ref{T:K-cw omega-cw odd}, we only need to show \eqref{E:hA-cw psc odd}.

If there exist no $\hat{A}$-admissible pairs over $M$, then $\hAcw(M)=0$, and the theorem is immediate.

If there does exist an $\hat{A}$-admissible pair over $M$, then there exists a compact subset $K\Subset M^\circ$ such that outside $K$, $E$ is isomorphic to the trivial bundle with trivial connection, and $\rho$ is locally constant. Let $E_0\to M$ denote the trivial flat bundle which is isomorphic to $E$ outside $K$ (so that $\ch(E_0)=\rk(E)$). Then $(E,E_0)$ forms a Gromov--Lawson pair with support $K$ in the sense of Assumption~\ref{A:GL pair}.

As in Subsection~\ref{SS:codim 0}, let $D=\slaD_E\oplus-\slaD_{E_0}$ be the Dirac operator on the twisted spinor bundle $S=(\slaS\otimes E)\oplus(\slaS\otimes E_0)$. Choose a compact subset $M'\Subset M$ with smooth boundary such that $M'$ contains $\pM$ and an open neighborhood of $K$. Consider the Callias operator $\calD_\Phi=D+\Phi\theta$ with the Callias potential $\Phi$ being a non-negative function such that $\Phi=0$ on $K$ and $\Phi=\delta$ near $\pM$ and in a neighborhood of $M\setminus M'$ for some constant $\delta>0$. As in the proof of Theorem~\ref{T:codim 0 spf},
\[
\spf(\calD_{\Phi,\id},\rho\oplus\rho)=\spf(\calD^{M'}_{\Phi,\id},\rho_{|M'}\oplus\rho_{|M'}).
\]
The right-hand side denotes the spectral flow corresponding to the restrictions of everything to $M'$.

Now $M'$ is a compact manifold, we can apply Corollary~\ref{IntroC:codim 0 spf} to obtain
\[
\spf(\calD^{M'}_{\Phi,\id},\rho_{|M'}\oplus\rho_{|M'})=\spf(D_{V(E,E_0)},\tilde{\rho}),
\]
where the right-hand side is a spectral flow on the closed manifold $\double M'=M'\cup_{\pM'}M'^-$, and $\tilde{\rho}$ is the gluing of $\rho_{|M'}$ and $\rho_{|M'}$. By the cohomological formula for spectral flow (cf. \cite[Theorem~2.8]{Getzler93}),
\[
\begin{aligned}
\spf(D_{V(E,E_0)},\tilde{\rho})& =\int_{\double M'}\hat{A}(\double M')\ch(V(E,E_0))\ch(\tilde{\rho}) \\
& =\int_M\hat{A}(M)(\ch(E)-\rk(E))\ch(\rho)\ne0,
\end{aligned}
\]
where the second equality is because of the opposite orientation on $M'^-$ and the triviality of $\rho$ outside $M'$, and the inequality is because of \eqref{E:omega admi pair}. In summary, we have shown that $\spf(\calD_{\Phi,\id},\rho\oplus\rho)\ne0$.

Consider
\[
\calD_\Phi(r)=(1-r)\calD_\Phi+r(\rho\oplus\rho)^{-1}\calD_\Phi(\rho\oplus\rho).
\]
For $u\in\dom(\calD_{\Phi,\id}(r))$, noticing that $\bfs\Phi>0$ on $\pM$, combined with \eqref{E:Callias square int-1}, \eqref{E:Lich for} and the hypothesis that $H_g\ge0$, we have
\[
\begin{aligned}
\int_M|\calD_\Phi(r)u|^2\D V_M=& \int_M\left(|\nabla(r)u|^2+\langle u,\Theta(r)u\rangle\right)\D V_M \\
& +\int_\pM\left\langle u,\left(\bfs\Phi+\frac{n-1}{2}H_g\right)u\right\rangle\D V_\pM \\
\ge& \int_M\langle u,\Theta(r)u\rangle\D V_M, 
\end{aligned}
\]
provided that the self-adjoint bundle endomorphism $\Theta(r)$ given by
\[
\begin{aligned}
\Theta(r):=& \frac{1}{4}\scal_g+\calR^{E\oplus E_0}(r)+\Phi^2+\rmc(\D\Phi)\theta \\
\ge& \frac{1}{4}\sigma^2n(n-1)+\calR^E-4r(1-r)\calR^\rho+\Phi^2-|\D\Phi|
\end{aligned}
\]
is non-negative (cf. Remark~\ref{R:Callias square int}).

Note that
\[
\|\calR^E-4r(1-r)\calR^\rho\|\le\frac{1}{2}n(n-1)\left(\|R^E\|+\|R^\rho\|\right).
\]
If one had $\|R^E\|+\|R^\rho\|<\frac{\sigma^2}{2}$, then one can always choose $\delta>0$ small enough such that $\Theta(r)$ is a uniformly positive operator for all $r\in[0,1]$, which makes $\calD_{\Phi,\id}(r)$ invertible, contradicting the fact that $\spf(\calD_{\Phi,\id},\rho\oplus\rho)\ne0$.

Therefore, it always holds that $\|R^E\|+\|R^\rho\|\ge\frac{\sigma^2}{2}$, and \eqref{E:hA-cw psc odd} follows by taking infimum and reciprocal. This completes the proof.
\end{proof}

\begin{remark}\label{R:K-cw psc odd-1}
Like \cite[Theorem~7]{BaerHanke24K-cw}, Theorem~\ref{T:K-cw psc odd} can be generalized to the version that $\sigma>0$ represents the lower bound of the self-adjoint curvature endomorphism in the Weitzenb\"ock formula of a general Dirac operator. In this case, $\omega$ will be chosen to be the differential form appearing in the cohomological formula of the corresponding spectral flow.
\end{remark}

\begin{remark}\label{R:K-cw psc odd-2}
In the case that $M$ has infinite \emph{uniform} K-cowaist, i.e., there exists a fixed compact subset $K\Subset M^\circ$ such that $\Kcw(M|K)$ is infinite, one can prove as before by choosing a suitable upper bound $\delta>0$ for the Callias potential $\Phi$ to show that $M$ cannot admit a complete metric of PSC and non-negative mean curvature along the boundary. A remaining question is whether this uniformity condition can be dropped.
\end{remark}

\subsection{Area enlargeable manifolds}\label{SS:area enlarg}

It is well known that an even-dimensional finitely area enlargeable manifold (cf. Definition~ \ref{IntroD:area enlarg}) has infinite K-cowaist. This is still true in odd dimensions.

\begin{proposition}\label{P:area enlarg K-cw odd}
An odd-dimensional finitely area enlargeable manifold in the sense of Definition~\ref{IntroD:area enlarg} has infinite K-cowaist, therefore has infinite $\omega$-cowaist for any $\omega$.
\end{proposition}

This proposition is a consequence of Proposition~\ref{P:K-cw covering}, Remark~\ref{R:K-cw subset} and the following lemma.

\begin{lemma}\label{L:contracting map K-cw}
Let $(M,g)$ be an oriented complete Riemannian $(2m-1)$-manifold possibly with compact boundary. If there exists a smooth $\varepsilon$-area-contracting map $f:M\to\rmS^{2m-1}(1)$ of non-zero degree which is locally constant at infinity and near the boundary, then
\[
\Kcw(M|K)\ge\frac{2}{\varepsilon},
\]
where $K=\supp(\D f)$.
\end{lemma}

\begin{proof}
Following \cite{LiSuWang24}, we consider the trivial vector bundle $E_0:=\rmS^{2m-1}(1)\times\slaS_{2m,+}$ over $\rmS^{2m-1}(1)$, where $\slaS_{2m,+}$ denotes the space of half-spinors on $\RR^{2m}$. There is a smooth function $\bar{\rho}:\rmS^{2m-1}(1)\to\End(\slaS_{2m,+})\cong U(2^{m-1})$ with the property that
\[
\int_{\rmS^{2m-1}(1)}\ch(\bar{\rho})=1\quad\mbox{and}\quad\|R^{\bar{\rho}}\|=\frac{1}{2}.\footnote{To be precise, this choice of $\bar{\rho}$ implies that
\[
(\bar{\rho}^{-1}(\D\bar{\rho}))^2(\epsilon_i,\epsilon_j)=2\bar{\rmc}(\epsilon_i^*)\bar{\rmc}(\epsilon_j^*),
\]
where $\bar{\rmc}(\cdot)$ is the Clifford multiplication on $\RR^{2m}$, and $\{\epsilon_i\}_{i=1}^{2m-1}$ is a local orthonormal frame of $T\rmS^{2m-1}(1)$.}
\]
Let $\rho:M\to U(2^{m-1})$ be the pullback $f^*(\bar{\rho})$. Then $\rho$ is locally constant at infinity and near the boundary, and
\[
\int_M\ch(\rho)=\deg(f)\int_{\rmS^{2m-1}(1)}\ch(\bar{\rho})\ne0.
\]
In this case, the condition \eqref{E:admi pair} is vacuous for $E$. Hence for any compatible bundle $E\to M$ with support $K$, $(E;\rho)$ is an admissible pair with support $K$. We can choose $E$ to be  a trivially flat bundle. Then
\[
\|R^E\|+\|R^\rho\|=\|f^*(R^{\bar{\rho}})\|\le\frac{\varepsilon}{2}.
\]
The lemma then follows from Definition~\ref{D:K-cw subset}.
\end{proof}

\begin{remark}\label{R:contracting map K-cw}
In fact, from \cite{LiSuWang24}, $\int_M\hat{A}(M)\ch(\rho)=\int_M\ch(\rho)\ne0$. So $(E;\rho)$ is also an $\hat{A}$-admissible pair. This indicates the same estimate for the $\hat{A}$-cowaist
\[
\hAcw(M|K)\ge\frac{2}{\varepsilon}.
\]
The same conclusion holds for even-dimensional manifolds (by choosing $E$ to be the pullback of the half spinor bundle over the sphere).
\end{remark}

\begin{remark}\label{R:area enlarg psc odd}
By Theorem~\ref{T:K-cw psc odd} and Proposition~\ref{P:area enlarg K-cw odd}, an odd-dimensional finitely area enlargeable spin manifold cannot admit a complete uniformly PSC metric. This fact may not be got conveniently by taking product with $\rmS^1$ and converting to even dimensions, due to the \emph{local constancy} assumption.
\end{remark}

\subsection{Codimension zero quantitative scalar curvature estimates}\label{SS:quant est codim 0}

By Theorems~\ref{IntroT:K-cw psc even} and \ref{T:K-cw psc odd}, on a complete Riemannian spin $n$-manifold $(M,g)$ without boundary, one has the estimate 
\begin{equation}\label{E:inf scal < 1/hatA-cw}
\inf(\scal_g)\le\frac{2n(n-1)}{\hAcw(M)}.
\end{equation}
We now present a quantitative version of this estimate.

\begin{theorem}\label{T:quant est codim 0}
Let $(M,g)$ be an $n$-dimensional complete Riemannian spin manifold without boundary. Let $K\Subset M$ be a compact subset of $M$. Suppose that
\begin{enumerate}
\item $\scal_g>\frac{2n(n-1)}{\hAcw(M|K)}$ on $K$; \label{IT:quant est codim 0-1}
\item $\scal_g\ge\sigma^2n(n-1)$ for some $\sigma>0$ on $K_\delta$, where $K_\delta$ is a $\delta$-neighborhood of $K$ with $\delta<\frac{\pi}{\sigma n}$. \label{IT:quant est codim 0-2}
\end{enumerate}
Then
\[
\inf_{p\in M}(\scal_g(p))<-\sigma^2n(n-1)\tan^2\Big(\frac{1}{2}\sigma n\delta\Big).
\]
\end{theorem}

\begin{proof}
First consider $n$ to be odd. Suppose, by contradiction, that
\begin{equation}\label{E:scal ge tan^2}
\scal_g\ge-\sigma^2n(n-1)\tan^2\Big(\frac{1}{2}\sigma n\delta\Big)\quad\mbox{on }M.
\end{equation}

We can assume $\hAcw(M|K)>0$, otherwise there is nothing to show. Denote
\[
\alpha:=\frac{1}{2n(n-1)}\inf_{p\in K}(\scal_g(p)),\qquad
\beta:=(\hAcw(M|K))^{-1}.
\]
So $\alpha>\beta$. Then there exists an $\hat{A}$-admissible pair $(E;\rho)$ with support $K$ such that
\begin{equation}\label{E:curv E+rho bound}
\alpha>\|R^E\|+\|R^\rho\|\ge\beta.
\end{equation}
Let $E_0$ be the trivial flat bundle which is isomorphic to $E$ outside $K$. So $(E,E_0)$ is a Gromov--Lawson pair with support $K$.

Construct a Lipschitz continuous function-valued Callias potential $\Phi:=h(x)$, where $h(t)=\frac{1}{2}\sigma n\tan(\frac{1}{2}\sigma nt)$ and
\[
x:M\to[0,\delta],\qquad x(p):=\min\{\dist_g(K,p),\delta\}.
\]
Then
\begin{equation}\label{E:Phi^2-dPhi bound}
\Phi^2-|\D\Phi|\ge-\frac{1}{4}\sigma^2n^2
\end{equation}
almost everywhere on $M$ and
\begin{equation}\label{E:Phi on M-Kdelta}
(\Phi^2-|\D\Phi|)_{|M\setminus K_\delta}=\Phi^2_{|M\setminus K_\delta}\equiv\frac{1}{4}\sigma^2n^2\tan^2\Big(\frac{1}{2}\sigma n\delta\Big).
\end{equation}
Consider the Callias operator $\calD_\Phi=D+\Phi\theta$ on $M$ as in Subsection~\ref{SS:codim 0} where $F=E_0$. We hope to compute the spectral flow $\spf(\calD_\Phi,\rho\oplus\rho)$. All the tools in Section~ \ref{S:spf Callias} can be used as pointed out in Remark~\ref{R:Lip potential}.

Choose a compact subset $M'\Subset M$ with smooth boundary such that $K_\delta$ is contained in the interior of $M'$. As before, cutting $M$ along $\pM'$, the spectral flow $\spf(\calD_\Phi,\rho\oplus\rho)$ splits into two parts. If we put the boundary condition given by $\bfs=\id$, then the part on $\overline{M\setminus M'}$ vanishes. For the remaining part $\spf(\calD^{M'}_{\Phi,\id},\rho_{|M'}\oplus\rho_{|M'})$, as in the proof of Theorem~\ref{T:K-cw psc odd}, by Corollary~\ref{IntroC:codim 0 spf} and the triviality of $\rho$ outside $M'$,
\[
\spf(\calD^{M'}_{\Phi,\id},\rho_{|M'}\oplus\rho_{|M'})=\spf(D_{V(E,E_0)},\tilde{\rho})=\int_M\hat{A}(M)(\ch(E)-\rk(E))\ch(\rho)\ne0,
\]
where $D_{V(E,E_0)}$ is the twisted Dirac operator on $\double M'$ and $\tilde{\rho}$ is the gluing of $\rho_{|M'}$ and $\rho_{|M'}$.
In summary, we get $\spf(\calD_\Phi,\rho\oplus\rho)\ne0$.

On the other hand, for $r\in[0,1]$, let
\[
\calD_\Phi(r)=(1-r)\calD_\Phi+r(\rho\oplus\rho)^{-1}\calD_\Phi(\rho\oplus\rho)=D(r)+\Phi\theta.
\]
For $u\in\dom(\calD_\Phi(r))$, by \eqref{E:Callias square int-2},
\[
\int_M|\calD_\Phi(r)u|^2\D V_M\ge\int_M\langle u,\Theta(r)u\rangle \D V_M,
\]
provided that the self-adjoint bundle endomorphism 
\[
\Theta(r):=\frac{n}{n-1}\calR(r)+\Phi^2-|\D\Phi|
\]
is non-negative,
where
\begin{equation}\label{E:calR bound}
\begin{aligned}
\calR(r)& =\frac{1}{4}\scal_g+\calR^E-4r(1-r)\calR^\rho \\
& \ge\frac{1}{4}\scal_g-\frac{n(n-1)}{2}(\|R^E\|+\|R^\rho\|).
\end{aligned}
\end{equation}
We examine $\Theta(r)$ that for any $r\in[0,1]$, \ref{IT:quant est codim 0-1}, \eqref{E:curv E+rho bound}, \eqref{E:calR bound} $\Longrightarrow$
\[
\Theta(r)=\frac{n}{n-1}\calR(r)>0\quad\mbox{on }K;
\]
\ref{IT:quant est codim 0-2}, \eqref{E:Phi^2-dPhi bound} $\Longrightarrow$
\[
\Theta(r)=\frac{n}{n-1}\cdot\frac{1}{4}\scal_g+\Phi^2-|\D\Phi|\ge0\quad\mbox{on }K_\delta\setminus K;
\]
\eqref{E:scal ge tan^2}, \eqref{E:Phi on M-Kdelta} $\Longrightarrow$
\[
\Theta(r)=\frac{n}{n-1}\cdot\frac{1}{4}\scal_g+\Phi^2\ge0\quad\mbox{on }M\setminus K_\delta.
\]

From above, if $u\in\ker(\calD_\Phi(r))$ for any $r\in[0,1]$, then $u=0$ on $K$ and $u$ satisfies the conditions of equality in the spectral estimates. From the discussion in \cite[Remark~4.5]{CeccZeid24GT}, one can deduce that $u$ vanishes almost everywhere on $M$, which means that $\spf(\calD_\Phi,\rho\oplus\rho)=0$, a contradiction! This proves the theorem for $n$ odd.

When $n$ is even, one only has a compatible bundle in the definition of $\hat{A}$-cowaist. In this case, one considers instead the index of a single Callias operator from the relative Dirac bundle of the form \cite[Example~2.5]{CeccZeid24GT} associated to a Gromov--Lawson pair. The computation and argument are essentially the same as above (and indeed a little simpler). To sum up, the theorem is proved.
\end{proof}

\begin{remark}\label{R:quan est codim 0}
We briefly sketch how \eqref{E:inf scal < 1/hatA-cw} can be concluded by Theorem~\ref{T:quant est codim 0}. Suppose \eqref{E:inf scal < 1/hatA-cw} does not hold, that is, $\inf(\scal_g)>\frac{2n(n-1)}{\hAcw(M)}$. By Remark~\ref{R:K-cw subset}, we can find a compact $K\Subset M$ such that
\[
\frac{2n(n-1)}{\inf(\scal_g)}<\hAcw(M|K)\le\hAcw(M).
\]
So Theorem~\ref{T:quant est codim 0}.\ref{IT:quant est codim 0-1} is satisfied. Theorem~\ref{T:quant est codim 0}.\ref{IT:quant est codim 0-2} can also be satisfied by choosing $\sigma$ and $\delta$ small enough. Then one concludes that $\inf(\scal_g)<0$, a contradiction.
\end{remark}

In the case involving area-contracting maps, we have the following  quantitative Llarull's theorem on non-compact manifolds as in \cite[Theorem~1.4]{Shi24LN} (where $\varepsilon=1$).

\begin{corollary}\label{C:quan est codim 0-contr}
Let $(M,g)$ be an $n$-dimensional complete Riemannian spin manifold without boundary. Let $f:M\to\rmS^n(1)$ be a smooth $\varepsilon$-area-contracting map of non-zero degree which is locally constant at infinity. Suppose that
\begin{enumerate}
\item $\scal_g\ge\varepsilon n(n-1)$ on $K:=\supp(\D f)$; \label{IC:quant est codim 0-1}
\item $\scal_g\ge\sigma^2n(n-1)$ for some $\sigma>0$ on $K_\delta$, where $K_\delta$ is a $\delta$-neighborhood of $K$ with $\delta<\frac{\pi}{\sigma n}$. \label{IC:quant est codim 0-2}
\end{enumerate}
Then
\[
\inf_{p\in M}(\scal_g(p))<-\sigma^2n(n-1)\tan^2\Big(\frac{1}{2}\sigma n\delta\Big).
\]
\end{corollary}

For the proof, note that $\hAcw(M|K)\ge\frac{2}{\varepsilon}$ by Lemma~\ref{L:contracting map K-cw} and Remark~\ref{R:contracting map K-cw}. So $\scal_g\ge\frac{2n(n-1)}{\hAcw(M|K)}$ on $K$. The $\hat{A}$-admissible pair $(E;\rho)$ (or the compatible bundle $E$ when $n$ is even) is chosen as in Remark~ \ref{R:contracting map K-cw}. A difference is that hypothesis \ref{IC:quant est codim 0-1} of Corollary~ \ref{C:quan est codim 0-contr} is not strict. But this does not matter since it is enough to show that the bundle endomorphism $\frac{n}{n-1}\calR(r)+\Phi^2-|\D\Phi|$ is non-negative everywhere and positive on some open subset of $K$. This is of course satisfied by the fact that the area-contracting constant of $f$ approaches zero near $\p K$.

The next is a result under the circumstance of infinite K-cowaist.

\begin{corollary}\label{C:quan est codim 0-inf}
Let $(M,g)$ be an $n$-dimensional complete Riemannian spin manifold without boundary. Let $K\Subset M$ be a compact subset of $M$ with smooth boundary. Suppose that
\begin{enumerate}
\item $\Kcw(M|K)=\infty$;
\item $\scal_g>0$ on $K$;
\item $\scal_g\ge\sigma^2n(n-1)$ for some $\sigma>0$ on an inward open geodesic collar neighborhood $U\subset K$ of $\p K$ of width $\delta\in(0,\frac{\pi}{\sigma n})$.
\end{enumerate}
Then
\[
\inf_{p\in M}(\scal_g(p))<-\sigma^2n(n-1)\tan^2\Big(\frac{1}{2}\sigma n\delta\Big).
\]
\end{corollary}

\begin{proof}
It is clear that Theorem~\ref{T:K-cw omega-cw odd} still holds for K-cowaist subject to a subset. So we have $\hAcw(M|K)=\infty$. By Theorem~\ref{T:quant est codim 0}, it suffices to construct a smaller compact subset $K'\Subset K^\circ$ such that $\hAcw(M|K')=\infty$. We proceed as in \cite[Proposition~6.4]{CeccZeid24GT}. Choose a smooth map $\psi:M\to M$ which is smoothly homotopic to the identity map such that $\psi_{|U}$ agrees with the projection onto $\p K$. Then for any $\hat{A}$-admissible pair $(E;\rho)$ with support $K$, $(\psi^*(E);\psi^*(\rho))$ is trivial outside $K':=K\setminus U$. By homotopy invariance of Chern character and odd Chern character, we know $(\psi^*(E);\psi^*(\rho))$ is still an $\hat{A}$-admissible pair (with support $K'$). Since $\psi$ has a bounded Lipschitz constant, pulling back by $\psi$ will not change the property of having infinite $\hat{A}$-cowaist. Therefore $\hAcw(M|K')=\infty$.
\end{proof}

\subsection{Codimension zero scalar-mean curvature estimates}\label{SS:scal-mean est codim 0}

It is natural to convert the previous quantitative estimates to compact manifolds with boundary. In this case they become quantitative width estimates.

The following is a scalar-mean curvature comparison theorem for the general long neck problem.

\begin{theorem}\label{T:scal-mean est codim 0}
Let $(M,g)$ be an $n$-dimensional compact Riemannian spin manifold with boundary. Let $K\Subset M^\circ$ be a compact subset. Suppose that
\begin{enumerate}
\item $\scal_g> \frac{2n(n-1)}{\hAcw(M|K)}$ on $K$;
\item $\scal_g\ge\sigma^2n(n-1)$ on $M\setminus K$ for some $\sigma>0$;
\item the mean curvature $H_g\ge-\sigma\tan(\frac{1}{2}\sigma nl)$ for some $l\in(0,\frac{\pi}{\sigma n})$.
\end{enumerate}
Then
\[
\dist_g(K,\pM)<l.
\]
\end{theorem}

\begin{proof}
The proof is essentially the same as that of Theorem~\ref{T:quant est codim 0}. See also \cite[Section~4]{Shi24LN} for details. Roughly speaking, find an $\hat{A}$-admissible pair $(E;\rho)$ with support $K$ satisfying \eqref{E:curv E+rho bound}. Consider directly the spectral flow $\spf(\calD_{\Phi,\id},\rho\oplus\rho)$, where the potential $\Phi$ is now defined from the distance function to $\pM$. On one hand, it is non-zero by the admissibility of $(E;\rho)$. On the other hand, one uses spectral estimate (on manifolds with boundary) to show that it must be zero. This finishes the proof.
\end{proof}

The following is a corresponding version to Corollary~\ref{C:quan est codim 0-contr} obtained in \cite{CeccZeid24GT,Shi24LN}.

\begin{corollary}\label{C:scal-mean est codim 0-contr}
Let $(M,g)$ be an $n$-dimensional compact Riemannian spin manifold with boundary. Let $f:M\to\rmS^n(1)$ be a smooth $\varepsilon$-area-contracting map of non-zero degree which is locally constant near the boundary. Suppose that
\begin{enumerate}
\item $\scal_g\ge\varepsilon n(n-1)$ on $\supp(\D f)$;
\item $\scal_g\ge\sigma^2n(n-1)$ on $M\setminus\supp(\D f)$ for some $\sigma>0$;
\item $H_g\ge-\sigma\tan(\frac{1}{2}\sigma nl)$ for some $l\in(0,\frac{\pi}{\sigma n})$.
\end{enumerate}
Then
\[
\dist_g(\supp(\D f),\pM)<l.
\]
\end{corollary}

The above width estimates are neck length estimates, which is about how long the distance between the boundary and a fixed compact subset can be under certain scalar and mean curvature conditions. The next width estimate, which corresponds to Corollary~\ref{C:quan est codim 0-inf}, is about how wide a geodesic collar neighborhood of the boundary can be. It generalizes \cite[Theorem~1.7]{CeccZeid24GT} to odd dimensions.

\begin{theorem}\label{T:scal-mean est codim 0-inf}
Let $(M,g)$ be an $n$-dimensional ($n\ge3$ odd) compact Riemannian spin manifold with boundary such that $\Kcw(M)=\infty$. Suppose $\scal_g>0$ and that there exist positive constants $\sigma$ and $l$ with $0<l<\frac{\pi}{\sigma n}$, such that
\[
H_g\ge-\sigma\tan\Big(\frac{1}{2}\sigma nl\Big).
\] 
Then $\pM$ admits no open geodesic collar neighborhood $U\subset M$ of width strictly greater than $l$ such that $\scal_g\ge\sigma^2n(n-1)$ on $U$.
\end{theorem}

\begin{proof}
We follow the idea of the proof of \cite[Theorem~1.7]{CeccZeid24GT}.
Let $U_d$ denote the open geodesic collar neighborhood of $\pM$ of width $d$ for $d>0$ small. Suppose, by contradiction, that for some $l<l'<\frac{\pi}{\sigma n}$, there exists $U_{l'}$ such that $\scal_g\ge\sigma^2n(n-1)$ on $U_{l'}$.

Fix a $\Lambda\in(l,l')$ and set $K_\Lambda:=M\setminus U_\Lambda$. From the proof of Corollary~\ref{C:quan est codim 0-inf}, $\hAcw(M|K_\Lambda)=\infty$. So we can choose an $\hat{A}$-admissible pair $(E;\rho)$ with support $K_\Lambda$ such that $\|R^E\|+\|R^\rho\|$ is small enough. Again we have a Callias operator $\calD_\Phi$ associated to a Gromov--Lawson pair as in Subsection~\ref{SS:codim 0}, such that $\spf(\calD_{\Phi,\id},\rho\oplus\rho)\ne0$. The Callias potential $\Phi$ is chosen as follows.

For $\delta\in(0,\frac{1}{2}\sigma n)$, consider the function $h_\delta(t):=\delta\tan(\delta t)$ with $t\in[0,\frac{\pi}{\sigma n})$. Set $\Phi:=h_\delta(x)$, where $x:M\to[0,\Lambda]$ is the distance function from $K_\Lambda$. Then $\Phi_{|K_\Lambda}=0$. By choosing $\delta$ close enough to $\frac{1}{2}\sigma n$, we can make
\[
\Phi_{|\pM}=h_\delta(\Lambda)>\frac{1}{2}\sigma n\tan\Big(\frac{1}{2}\sigma nl\Big).
\]

For $r\in[0,1]$, $u\in\dom(\calD_{\Phi,\id}(r))$, in the estimate \eqref{E:Callias square int-2} for $\int_M|\calD_\Phi(r)u|\D V_M$, our choice guarantees that the boundary term
\[
\int_\pM\Big(\frac{1}{2}nH_g+\bfs\Phi\Big)|u|^2\D V_\pM\ge0.
\]
So
\[
\int_M|\calD_\Phi(r)u|\D V_M\ge\int_M\langle u,\Theta(r)u\rangle\D V_M,
\]
with the self-adjoint bundle endomorphism $\Theta(r)$ satisfies
\[
\Theta(r)\ge\frac{n}{n-1}\Big(\frac{1}{4}\scal_g-\frac{1}{2}n(n-1)(\|R^E\|+\|R^\rho\|)\Big)+\Phi^2-|\D\Phi|.
\]
On $K_\Lambda$, since $\scal_g$ has a positive lower bound and $\|R^E\|+\|R^\rho\|$ can be small enough,
\[
\Theta(r)\ge\frac{n}{n-1}\Big(\frac{1}{4}\scal_g-\frac{1}{2}n(n-1)(\|R^E\|+\|R^\rho\|)\Big)>0.
\]
On $M\setminus K_\Lambda$, since $\scal_g\ge\sigma^2n(n-1)$ and $\Phi^2-|\D\Phi|\ge h_\delta^2(t)-|h_\delta'(t)|=-\delta^2>-\frac{1}{4}\sigma^2n^2$,
\[
\Theta(r)\ge\frac{n}{n-1}\cdot\frac{1}{4}\scal_g+\Phi^2-|\D\Phi|>0.
\]
Hence $\calD_{\Phi,\id}(r)$ is invertible for any $r\in[0,1]$, contradicting $\spf(\calD_{\Phi,\id},\rho\oplus\rho)\ne0$.
\end{proof}

\subsection{Codimension one scalar-mean curvature estimate}\label{SS:scal-mean est codim 1}

In this subsection, we prove the scalar-mean curvature estimate for the band width problem stated in Theorem~\ref{IntroT:scal-mean est codim 1}.

\begin{proof}[Proof of Theorem~\ref{IntroT:scal-mean est codim 1}]
Actually infinite $\hat{A}$-cowaist of $N$ suffices, which follows from Theorem~\ref{T:K-cw omega-cw odd}. Thus for any $\varepsilon>0$, there exists an $\hat{A}$-admissible pair $(E_N;\rho_N)$ over $N$ such that $\|R^{E_N}\|+\|R^{\rho_N}\|<\varepsilon$. Pulling back to $M$ (up to diffeomorphism), one gets a Hermitian vector bundle $E\to M$ and a smooth map $\rho:M\to U(l)$ such that $\|R^E\|+\|R^\rho\|<c\,\varepsilon$ for some constant $c>0$.

Let $\uS:=\slaS\otimes E$, where $\slaS\to M$ is the complex spinor bundle over $M$, and form the Dirac bundle $S=\uS\oplus\uS$ as in Subsection~\ref{SS:codim 1}. Consider the Callias operator $\calD_\Phi$ of \eqref{E:Callias even} with the boundary condition given by $\bfs=\pm\id$ on $\p_\pm M$. Then by Corollary~ \ref{C:codim 1 spf} and Remark~\ref{R:codim 1 spf}, for any Callias potential $\Phi$, we have
\[
\spf(\calD_{\Phi,\bfs},\rho)=2\spf(\slaD_{E_N},\rho_N)=2\int_N\hat{A}(N)\ch(E_N)\ch(\rho_N)\ne0,
\]
where $\slaD_{E_N}$ is the spin Dirac operator on $N$ twisted by $E_N$, and the inequality is due to the fact that $(E_N,\rho_N)$ is an $\hat{A}$-admissible pair over $N$.

The next step is routinely choosing a suitable Callias potential $\Phi$ so that $\spf(\calD_{\Phi,\bfs},\rho)$ vanishes. The selection of the Callias potential is essentially the same as in the proof of \cite[Theorem~7.6]{CeccZeid24GT} and the argument is the same as those in previous subsections, so we do not repeat here.
\end{proof}

\begin{remark}\label{R:scal-mean est codim 1}
Typical examples of closed manifolds with infinite $\hat{A}$-cowaist include finitely area enlargeable manifolds, their product with manifolds of non-zero $\hat{A}$-genus, and manifolds which admit a smooth map to a finitely area enlargeable manifold of non-zero $\hat{A}$-degree.

Like \cite[Section~7]{CeccZeid24GT}, Theorem~\ref{IntroT:scal-mean est codim 1} will hold for more general compact Riemannian band $(M,g)$ having \emph{infinite vertical $\hat{A}$-cowaist}, meaning that for any $\varepsilon>0$, there exist a Hermitian bundle $E\to M$ and a smooth map $\rho:M\to U(l)$ such that $\|R^E\|+\|R^\rho\|<\varepsilon$ and
\[
\int_{\p_+M}\hat{A}(\p_+M)\ch(E_{|\p_+M})\ch(\rho_{|\p_+M})\ne0.
\]
\end{remark}

The following are two consequences of Theorem~\ref{IntroT:scal-mean est codim 1}.

\begin{corollary}\label{C:mean curv <0}
Let $N$ be a closed $(n-1)$-dimensional ($n\ge2$ even) spin manifold of infinite $\hat{A}$-cowaist and let $M=N\times[-1,1]$. If there exists a Riemannian metric $g$ on $M$ such that $\scal_g>0$, then $\inf_{p\in\pM}(H_g(p))<0$.
\end{corollary}

\begin{corollary}\label{C:band width}
Let $N$ be a closed $(n-1)$-dimensional ($n\ge2$ even) spin manifold of infinite $\hat{A}$-cowaist and let $M=N\times[-1,1]$. If there exists a Riemannian metric $g$ on $M$ such that $\scal_g\ge\sigma^2n(n-1)$ for some $\sigma>0$, then $\wid(M,g)<\frac{2\pi}{\sigma n}$.
\end{corollary}

Corollary~\ref{C:band width} further implies the following non-existence result about uniformly PSC metrics.

\begin{corollary}\label{C:non-exist NxR}
Let $N$ be a closed $(n-1)$-dimensional ($n\ge2$ even) spin manifold of infinite $\hat{A}$-cowaist. Then $N\times\RR$ does not admit a complete metric of uniformly positive scalar curvature.
\end{corollary}

\begin{remark}\label{R:band width & non-exist NxR}
It is shown in \cite{Gromov18metric} by Gromov that the width estimate of Corollary~\ref{C:band width} is optimal. Starting from \cite{Gromov18metric}, the band width problem in the background of positive scalar curvature has been studied extensively recently, for example, in \cite{Zeidler22,Cecchini20LN,Zeidler20,GuoXieYu23band,CeccZeid23,CeccZeid24GT,Rade23,KumarSen25}, etc.
In particular, Corollary~\ref{C:band width} holds when $N$ is a closed spin manifold of non-vanishing Rosenberg index.

For Corollary~\ref{C:non-exist NxR}, it has also been proved in \cite{Cecchini20C*} without the uniformity restriction, for $N$ being a closed spin manifold of non-vanishing Rosenberg index. This is related to a conjecture of Rosenberg and Stolz \cite[Conjecture~7.1]{RosenbergStolz94} which asserts that if $N$ is a closed manifold which does not admit a PSC metric, then $N\times\RR$ does not admit a complete PSC metric.

From the above perspective, it would be natural to ask whether a closed spin manifold with infinite K-cowaist has non-vanishing Rosenberg index. This is in the spirit of Schick's conjecture \cite[Conjecture~ 1.5]{Schick14ICM} which says that ``every obstruction to positive scalar curvature for manifolds of dimension $\ge5$ which is based on index theory of Dirac operators can be read off the Rosenberg index''. Note that the method of Hanke and Schick in \cite{HankeSchick06,HankeSchick07} may be helpful, where they proved that closed area enlargeable spin manifolds in the sense of \cite{GromovLawson83} have non-vanishing Rosenberg index.
\end{remark}


\bibliographystyle{amsplain}

\end{document}